\documentclass[11pt]{amsart}
\usepackage{fullpage,amsmath,amsthm,amsfonts,amssymb,graphicx,amscd,float,picins,hyperref,
	setspace}
\usepackage[shortlabels]{enumitem}
\usepackage{color, xcolor}
\usepackage{tikz-cd}
\usepackage{xypic}
\usetikzlibrary{cd}
\usepackage{wrapfig}

\definecolor{ivory}{rgb}{1.0, 1.0, 0.95}
\definecolor{carnelian}{rgb}{0.7, 0.11, 0.11}
\definecolor{cinnamon}{rgb}{0.82, 0.41, 0.12}

\title{IW4}

\newtheorem{thm}{Theorem}[section]

\newtheorem{lem}[thm]{Lemma}

\newtheorem{prop}[thm]{Proposition}

\theoremstyle{definition}

\newtheorem{df}[thm]{Definition}

\newcommand{\out}{\textup{Out}(F_r)}
\newcommand{\outt}{\textup{Out}(F_3)}

\newcommand{\os}{\textup{CV}_r}
\newcommand{\ost}{\textup{CV}_3}

\newcommand{\teich}{Teichm\"{u}ller }
\newcommand{\RR}{\mathbb R}
\newcommand{\NN}{\mathbb N}

\newcommand{\ol}{\overline}

\newcommand{\mA}{\mathcal{A}}

\newcommand{\mV}{\mathcal{V}}
\newcommand{\mD}{\mathcal{D}}

\newcommand{\mM}{\mathcal{M}}
\newcommand{\G}{\Gamma}
\newcommand{\ZZ}{\mathbb{Z}}

\newcommand{\from}{\colon}

\newcommand{\vphi}{\varphi}

\newcommand{\cv}{\os}

\DeclareMathOperator{\LW}{LW}
\DeclareMathOperator{\SW}{SW}
\DeclareMathOperator{\IW}{IW}

\theoremstyle{definition}

\theoremstyle{plain}

\newtheorem{thmx}{Theorem}

\newtheoremstyle{TheoremNum}
{8.0pt plus 2.0pt minus 4.0pt}{8.0pt plus 2.0pt minus 4.0pt} %%% space between body and thm
{\itshape} %%% Thm body font
{-0.15cm} %%% Indent amount (empty = no indent)
{\bfseries} %%% Thm head font
{.} %%% Punctuation after thm head
{ }  %%% Space after thm head
{\thmname{#1}\thmnote{ \bfseries #3}}%%% Thm head spec
\theoremstyle{TheoremNum}
\newtheorem{duplicate}{}

\newcounter{commentcounter}

\usepackage{tikz}
\usetikzlibrary{decorations.markings}
\tikzset{middlearrow/.style n args={4}{
		decoration={
			markings,
			mark=at position #1 with {\arrow{#2},\node[transform shape,#4] {#3};}},postaction={decorate}},
	middlearrow/.default={.5}{>}{}{below}	}
\usetikzlibrary{calc}

\begin{document}

	\title{Low complexity among\\ principal fully irreducible elements of $\mathrm{Out}(F_3)$}
	\author{Naomi Andrew}
	\author{Paige Hillen}
	\author{Robert Alonzo Lyman}
	\author{Catherine Eva Pfaff}
	
	\address{\tt Mathematical Institute, University of Oxford 
		\newline \indent  {\url{https://naomigandrew.wordpress.com/}}, } \email{\tt Naomi.Andrew@maths.ox.ac.uk}
	
	\address{\tt University of California - Santa Barbara Department of Mathematics
		\newline \indent  {\url{https://sites.google.com/view/paigehillen/home}}, } \email{\tt paigehillen@ucsb.edu}
	
	\address{\tt Department of Mathematics, Rutgers University - Newark
		\newline \indent  {\url{https://ryleealanza.org/}}, } \email{\tt robbie.lyman@rutgers.edu}
	
	\address{\tt Department of Mathematics \& Statistics, Queen's University
		\newline \indent  {\url{https://mast.queensu.ca/~cpfaff/}}, } \email{\tt c.pfaff@queensu.ca}

	\date{}
	\maketitle

\begin{abstract}
	We find the shortest realized stretch factor for a fully irreducible $\vphi\in\outt$ and show that it is realized by a ``principal'' fully irreducible element. We also show that it is the only principal fully irreducible outer automorphism in any rank produced by a single fold. 
\end{abstract}

\section{Introduction}

Let $F_r$ denote the free group of rank $r \geq 2$, and consider its outer automorphism group $\out$. Each element of $\out$ has an associated stretch factor (also called its growth rate or dilatation) \[ \sup_{w \in F_r} \limsup \sqrt[n]{\| \varphi^n(w) \|}.\] Here $\|w\|$ denotes the cyclically reduced word length with respect to some fixed free basis of $F_r$.
The stretch factor records how fast elements grow under iteration of $\varphi$, and is independent of the chosen basis and the representative of the outer class. 

For a fully irreducible element, where no power $\varphi^k$ preserves a proper free factor, Bestvina and Handel \cite{bh92} construct an ``irreducible train track representative'' for $\varphi$, a self homotopy equivalence of a graph with good behavior under iteration and which induces $\varphi$ on the fundamental group. They show that the stretch factor of $\varphi$ is realised as the Perron--Frobenius eigenvalue of the transition matrix, a non-negative integer matrix associated to the train track representative. 

Here we are concerned with two aspects of this theory. The first aspect is the minimal stretch factors attained by fully irreducible outer automorphisms. The set of possible values is studied, for example, in \cite{AKRstretch}, \cite{dkl15}, \cite{dkl17}, and \cite{thurstonstretch} (exposited upon in \cite{thurstonexplained}). 
The second aspect explores properties of train track representatives for certain fully irreducible outer automorphisms, and the dynamics of their action on the Culler--Vogtmann Outer space $\os$.

In both aspects, we investigate the outer automorphism 
%$\psi\in\outt$ represented by 
\[
\psi=
\begin{cases}
	x \mapsto y \\
	y \mapsto z \\
	z \mapsto zx^{-1}. 
\end{cases}\]

\noindent  With a suitable marking, it is represented by the train track map:

\begin{center}
\includegraphics[width=2.25in]{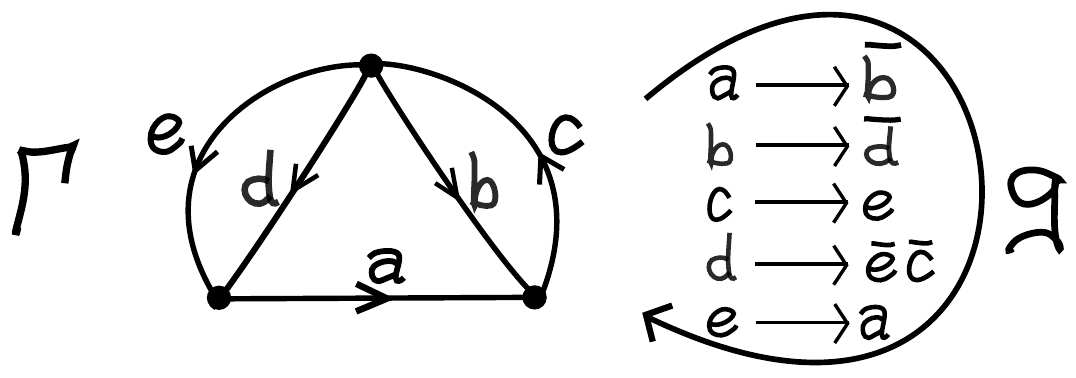}
\end{center}

Its stretch factor is the real root of $x^5-x-1$, approximately $1.167$. In fact, $\psi$ is a fully irreducible outer automorphism, and we show: 

\begin{duplicate}[Theorem~\ref{t:minimal}]
	The stretch factor of $\psi$ is minimal among fully irreducible elements in $\outt$. 
\end{duplicate}

The ingredients in the Theorem~\ref{t:minimal} proof are Lemma~\ref{lemma:Perron}, showing stretch factors of fully irreducible outer automorphisms are Perron numbers, and verification that $\frak{g}$ is an irreducible train track representative of a fully irreducible outer automorphism realising the minimum achievable Perron number. 

We use a result of Boyd listing the smallest Perron numbers for several fixed algebraic degrees \cite{boydsupplement}. Boyd lists the smallest (and in some degrees the second, third, or fourth smallest) Perron numbers for fixed algebraic degrees between 2 and 12. For ranks $r > 3$, this strategy of checking the smallest Perron numbers with algebraic degree in a fixed range will likely be less fruitful. The known examples of fully irreducible outer automorphisms with small stretch factors are quite a bit larger than the smallest Perron numbers listed in \cite{boydsupplement}.

Given a train track map, we can factor it as a series of \emph{folds} followed by a homeomorphism. It is notable that the train track map $\mathfrak{g}$ representing $\psi$ in Theorem~\ref{t:minimal} has a fold decomposition consisting of a single fold and a homeomorphism. As folds increase the number of edges in the image of a train track map, it seems reasonable to conjecture that outer automorphisms with minimal stretch factors have train track representatives with few folds in their fold decomposition. This idea is followed up and quantified in \cite{PaigeSymmetries}.

In the mapping class group setting, minimal stretch factors among pseudo-Anosov elements on a fixed surface of finite type have been widely studied. Until very recently, the minimum was only known for a few isolated examples of surfaces. Consider an orientable surface of finite type with genus $g$ and $p$ punctures. Previously the minimum stretch factor among all pseudo-Anosov mapping classes was precisely known only for small values of $(g, p)$, for example $(1,0)$ and $(1,1)$, $(2,0)$ \cite{ChoHam}, and $(0,p)$ for $p=3,4,5,6,7,8$ (\cite{SKL}, \cite{HamSong}, \cite{lt11braids}). Recent breakthrough work of \cite{TZ24} also computes the minimum stretch factor for $(0, p)$ for large $p$. Among orientable pseudo-Anosov mapping classes, the minimum is known for surfaces with $0$ punctures and genus $g=2,3,4,5,8$ (\cite{lt11}, \cite{h10}). For surfaces with $p > 0$ punctures, pseudo-Anosov mapping classes induce outer automorphisms of $F_r$ when $r = 2g+p-1$. However, there are many outer automorphisms which do not come from surface homeomorphisms, so minimal stretch factor results in the mapping class group setting do not directly carry over to $\out$. 

We also prove that the map $\psi$ is \emph{principal} in the sense of \cite{stablestrata}, where it is also proved that they possess a stability property mimicking that held by pseudo-Anosov mapping classes and used in \cite{randomout} and \cite{hittingmeasureout} to understand a typical outer automorphism and tree in $\partial\os$. Their action on Culler--Vogtmann Outer space generally more closely resembles a hyperbolic setting: While all fully irreducible elements act loxodromically, they only have an axis \emph{bundle} rather than a unique axis \cite{hm11}. But principal fully irreducible outer automorphisms do have a unique axis, reflecting the uniqueness of the Stallings fold decompositions of their irreducible train track representatives. This axis passes through the highest dimensional simplices of Outer space. 

The name ``principal'' is chosen to reflect similarities with the principal pseudo-Anosov mapping classes of a surface, as used for example by Masur in \cite{m82}. Every pseudo-Anosov mapping class has a representative leaving invariant a pair of transverse measured singular minimal foliations, one of which is expanded and the other of which is contracted by the homeomorphism. Pseudo-Anosov mapping classes act loxodromically on \teich space, and the pair of transverse foliations provides the endpoints of its axis in the Thurston boundary.
A pseudo-Anosov is \emph{principal} if these foliations have only 3-pronged singularities. Gadre and Maher \cite{gm16} proved principal pseudo-Anosov mapping classes random-walk generic. The ``singularity structure'' of a fully irreducible outer automorphism, namely its ideal Whitehead graph, is similarly controlled by its attracting endpoint, and in the case of a principal element it will be a union of the maximal number of triangles. As such, principal elements comprise the subset of those proved generic in \cite{randomout} that possess the stability property of \cite{stablestrata}.

As mentioned above, each train track map can be factored as a series of Stallings folds followed by a homeomorphism. Up to a reasonable equivalence relation, the number of folds in this decomposition is an invariant of a principal fully irreducible outer automorphism, and so we can try to minimise it. It turns out that the same outer automorphism achieves the minimum complexity --- over all ranks --- in this sense.

\begin{duplicate}[Theorem~\ref{t:unique}]
	Up to edge relabeling, the map $\mathfrak{g}$ is the only train track map representing a principal fully irreducible outer automorphism in any rank whose Stallings fold decomposition consists of only a single fold and then a graph-relabeling isomorphism.
\end{duplicate}

The tool we use for characterising rank-3 principal fully irreducible outer automorphisms is a directed graph $\widehat{\mA_3}$, we call the \emph{principal stratum automaton} (defined in \S \ref{ss:automaton}). This is a refinement of the lonely direction automaton of \cite{automaton}, which was used to determine the graphs carrying train track representatives of principal fully irreducible $\vphi\in\outt$. That theorem was proved by passing to a \emph{rotationless} power of $\varphi$, which is unhelpful for minimising other kinds of complexity such as the stretch factor or length of a Stallings fold decomposition. Our refinement allows us to work with the outer automorphism as it comes, allowing for the periodic behavior absent in the rotationless case.

\begin{duplicate}[Theorem~\ref{p:PrincipalGivesLoop}]
	Suppose $g$ is a train track representative of a principal fully irreducible $\vphi\in\outt$.
	Then the Stallings fold decomposition of $g$ is partial-fold conjugate to one determining a directed loop in $\widehat{\mA_3}$.
\end{duplicate}

We note that, subsequent to the posting of this paper, Theorem~\ref{p:PrincipalGivesLoop} was significantly generalized in \cite{pfaff2024out}.

We further note that, using \cite[Theorem 4.7]{loneaxes}, any principal fully irreducible $\vphi\in\out$ has a single unique invariant axis and the single unique invariant axis of any $\psi^{-1}\circ\vphi\circ\psi$ differs from that of $\vphi\in\out$ only by a change of marking. With a little work, this implies that the number of folds in a Stallings fold decompositions of the conjugacy class of a principal fully irreducible $\vphi$ is an invariant of the conjugacy class. This similarly holds for any ``lone axis'' fully irreducible outer automorphism. More generally, one would need to take the minimal number of folds in a fold decomposition.

In the process of proving Theorem~\ref{p:PrincipalGivesLoop}, we also prove in Proposition~\ref{lemma:map3} that any principal axis in $\ost$ passes through the $\outt$ orbit of simplices given by the graph shown above.

\subsection*{Structure of the paper} 
In \S\ref{ss:automaton} we introduce the principal stratum automaton, and prove that it captures the principal fully irreducible elements of $\outt$. \S\ref{s:ourphi} introduces the map $\frak{g}$ which is the subject of both theorems, and proves that it induces a principal fully irreducible outer automorphism. \S\ref{s:perron} contains the proof of Theorem~\ref{t:minimal} and \S\ref{s:single_fold} the proof of Theorem~\ref{t:unique}.
%; they are independent of one another.

\subsection*{Acknowledgements}
We are grateful to the Women in Groups Geometry and Dynamics (WiGGD) program from which this paper arose, and to Anna Parlak for discussions on related questions. We are grateful to Lee Mosher for providing his talents as a rabbit-hole preventing sounding board, and to Ilya Kapovich and Caglar Uyanik for helping us track down references and their ongoing interest in our work. We would also like to thank the referee for a careful reading. This work has received funding from the European Research Council (ERC) under the European Union's Horizon 2020 research and innovation programme (Grant agreement No. 850930), an NSF postdoctoral fellowship, and an NSERC Discovery Grant. The 4th author is grateful to the Institute for Advanced Study, and Bob Moses for funding her membership there.

%\bigskip

%%%%%%%%%%%%%%%%%%%%%%%%%%%%%%%%%%%%%%%%%%%%%%%%%%%%%%%%%%%%%%%%%%%%
%%%%%%%%%%%%%%%%%%%%%%%%%%%%%%%%%%%%%%%%%%%%%%%%%%%%%%%%%%%%%%%%%%%%

\section{Background}{\label{s:background}}

Assume throughout this section that $\G$ is a finite oriented graph where each vertex has valence at least $3$ and $F_r$ is a free group of rank $r\geq 3$.

%%%%%%%%%%%%%%%%%%%%%%%%%%%%%%%%%%%%%%%%%%%%%%%%%%%%%%%%%%%%%%%%%%%%%%%%

\subsection{Edge Maps on Graphs}{\label{s:graphmaps}}

Suppose $\G$ has positively oriented edges $\{e_1, e_2 \dots, e_n\}$ and vertices $\{v_1, v_2, \dots, v_m\}$. We use the notation $E\G:=\{e_1, e_2 \dots, e_n\}$, and $V\G :=\{v_1, v_2, \dots, v_m\}$, and $E^{\pm}\G = \{e_1, \overline{e_1}, \dots, e_n, \overline{e_n}\}$, with an overline indicating a reversal of orientation. Given $v\in V\G$, a \textit{direction at $v$} will mean an element of $E^{\pm}\G$ with initial vertex $v$. We let $\mD\G$ denote the set of directions at vertices in $\G$. A \textit{turn at $v$} will mean an unordered pair $\{d_1, d_2\}$ of directions at $v$. The turn is \emph{degenerate} if $d_1=d_2$.

\vspace*{-1.5mm}

An \textit{edge path} (or \textit{path}) $\rho$ in $\G$ is a finite sequence $(a_1, a_2, \dots, a_{\ell})\in (E^{\pm}\G)^{\ell}$ such that there exists a sequence $(v_1, v_2, \dots, v_{\ell-1})\in (V\G)^{\ell-1}$ satisfying that the turn $\{\overline{a_{j}}, a_{j+1}\}$ 	
\parpic[r]{\includegraphics[width=1.2in]{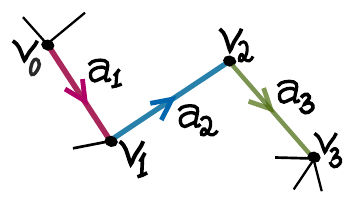}} 
\noindent is a turn at $v_j$ for each $j\in\{1,2,\dots,\ell-1\}$. For such a path $(a_1, a_2, \dots, a_m)$ we write $\gamma = a_1a_2\dots a_m$ and say $\gamma$ \textit{contains} the oriented edges $a_1, a_2, \dots, a_m$ and \textit{takes} the turns $\{\overline{a_1}, a_2\}$, $\{\overline{a_2}, a_3\}$, $\dots$, $\{\overline{a_{n-1}}, a_n\}$. We call $\gamma$ \textit{tight} if it takes no degenerate turns, which we colloquially describe as there being no ``backtracking.''

An \textit{edge (or graph) map} $g:\G\to\G'$ will mean
\begin{itemize}
\item a map $\mV:V\G\to V\G'$, where we write $g(v)$ for $\mV(v)$, together with
\item for each $e\in E^{\pm}\G$, an assignment of a path $g(e)$ in $\G'$ such that
\begin{enumerate}
\item if the initial vertex of $e$ is $v$, then the initial vertex of $g(e)$ is $g(v)$, and
\item if $g(e)$ is the edge path $g(e) = a_1a_2\dots a_m$, then $g(\overline{e})$ is the edge path concatenation $g(\overline{e}) = \overline{a_m}\dots\overline{a_2}~\overline{a_1}$.
\end{enumerate}
\end{itemize}

Viewing $\Gamma$ and $\Gamma'$ as topological spaces, $g$ is a continuous map sending vertices to vertices. We say a turn $\{d_1,d_2\}$ is $g$-taken if it appears in the image $g(e)$ of some edge of $\Gamma$, and call $g$ \emph{tight} if the image of each edge is a tight path. In particular, no degenerate turns are $g$-taken.

If $\gamma = a_1a_2\dots a_n$ is a path in $\G$ for some $a_1, a_2, \dots, a_n \in E^{\pm}\G$, then $g(\gamma)$ will mean the concatenation of edge paths $g(\gamma) = g(a_1)g(a_2)\dots g(a_n)$. Note that $g(\gamma)$ is tight if and only if $\gamma$ is tight and $g$ is locally injective on $\gamma$.

To $g$ we associate a \emph{direction map} $Dg:\mD\G\to\mD\G'$ such that if $g(e) = a_1 a_2\dots a_m$, for some $m\geq 1$ and $a_1,a_2,\dots,a_m \in E^{\pm}\G$, then $Dg(e) = a_1$.  Now suppose that $g\from \G \to \G$ is a self-map.
We call a direction $e$ \emph{periodic} if $Dg^k(e) = e$ for some $k>0$, and \emph{fixed} if $k=1$. The turn $\{d_1,d_2\}$ is called an \emph{illegal turn} for $g$ if $\{Dg^k(d_1),Dg^k(d_2)\}$ is degenerate for some $k$. Defining an equivalence relation on $\mD\G$ by $d_1\sim d_2$ when $\{d_1,d_2\}$ is an illegal turn, the equivalence classes are called \emph{gates}. Note that each gate at each periodic vertex contains a unique periodic direction.

Viewing $g$ as a continuous map of graphs, $g$ \emph{represents} $\vphi$ when $\pi_1(\G)$ has been identified with $F_r$ (that is $\G$ is \emph{marked}) and $\vphi$ is the induced map of fundamental groups. When a marking is not explicitly given, we mean ``there exists a marking such that.''

%%%%%%%%%%%%%%%%%%%%%%%%%%%%%%%%%%%%%%%%%%%%%%%%%%%%%%%%%%%%%%%%%%%%%%%%

\subsection{Train track maps and fully irreducible outer automorphisms}{\label{s:fi}}
	Suppose $g:\G\to\G$ is an edge map. We call $g$ a \textit{train track (tt) map} if $g^k$ is tight for each $k\in\ZZ_{>0}$. We call the train track map $g$ \textit{expanding} if for each edge $e \in E\G$ we have $|g^n(e)|\to\infty$ as $n\to\infty$, where for a path $\gamma$ we use $|\gamma|$ to denote the number of edges $\gamma$ traverses (with multiplicity). Note that, apart from our not requiring a ``marking,'' these definitions coincide with those in  \cite{bh92} when $g$ is in fact a homotopy equivalence of graphs (viewed topologically). 

The \textit{transition matrix} $M(g)$ of a tt map $g\from\Gamma \to \Gamma$ is the square $|E\G| \times |E\G|$ matrix $[a_{ij}]$ such that $a_{ij}$, for each $i$ and $j$, is the number of times $g(e_i)$ contains either $e_j$ or $\overline{e_j}$. Note that each transition matrix is a nonnegative integer matrix.

A nonnegative integral matrix $A=[a_{ij}]$ is \textit{irreducible} if for each $(i,j)$, there is a $k\in\ZZ_{>0}$ so that the ${ij}^{th}$ entry of $A^k$ is positive, and so in particular is at least $1$.
If $A^k$ is strictly positive for some $k\in\ZZ_{>0}$ then $A$ is \textit{primitive}. Furthermore, $A$ is \textit{Perron--Frobenius (PF)} if there exists an $N$ such that, for each $k \geq N$, we have that $A^k$ is strictly positive. 

For nonnegative integral matrices, being primitive is equivalent to being irreducible plus aperiodic. 
If $M=M(g)$ is primitive, then by Perron--Frobenius theory, $M$ has a unique eigenvalue $\lambda(g)$ of maximal modulus. $\lambda(g)\in\RR_{>1}$ and is the Perron--Frobenius (PF) eigenvalue of $M$. It is called the \textit{stretch factor} (or \textit{dilatation}) of $g$. 

A tt map is \textit{irreducible} if its transition matrix is irreducible. Not every element of $\out$ is represented by a tt map, and even fewer by irreducible tt maps. An outer automorphism $\vphi\in\out$ is \emph{fully irreducible} if no positive power preserves the conjugacy class of a proper free factor of $F_r$. Bestvina and Handel \cite{bh92} proved that each fully irreducible outer automorphism admits expanding irreducible tt representatives and that the stretch factor of these representatives is an invariant of the outer automorphism.

By \cite[Corollary 4.43]{fh11}, for each $r \geq 2$, there exists an $R(r) \in \NN$ such that for each expanding irreducible tt representative $g$ of a fully irreducible $\vphi \in \out$, among other properties, each periodic direction is fixed by $g^{R(r)}$. This power $R$ is called the \emph{rotationless} power.

%%%%%%%%%%%%%%%%%%%%%%%%%%%%%%%%%%%%%%%%%%%%%%%%%%%%%%%%%%%%%%%%%%%%
%%%%%%%%%%%%%%%%%%%%%%%%%%%%%%%%%%%%%%%%%%%%%%%%%%%%%%%%%%%%%%%%%%%%

\subsection{Whitehead graphs \& lamination train track (ltt) structures}{\label{ss:WGs}}
Local Whitehead graphs, stable Whitehead graphs, and ideal Whitehead graphs were introduced in \cite{hm11}.
We give definitions here only in the circumstance of no periodic Nielsen paths (PNPs), as this will always be the case for us. (PNPs only impact the ideal Whitehead graph definition; since we do not explicitly use them, we refer the reader to \cite{bh92,bfh00} for definitions.)

Let $g:\G \to \G$ be a tt map. The \textit{local Whitehead graph} $\LW(g; v)$ at a $v \in V\G$ has a vertex for each direction at $v$ and an edge connecting the vertices corresponding to a pair of directions $\{d_1,d_2\}$ at $v$ precisely when the turn $\{d_1,d_2\}$ is $g^k$-taken for some $k\in\ZZ_{>0}$. Given a fixed vertex $v$, the \emph{stable Whitehead graph} $\SW(g;v)$ is the restriction of $\LW(g;v)$ to the periodic direction vertices and the edges connecting them. In terms of gates, $\SW(g;v)$ has a vertex for each gate at $v$.
	
In the absence of PNPs, if $g$ represents a fully irreducible outer automorphism $\vphi$, then the \textit{ideal Whitehead graph} $\IW(\vphi)$ for $\vphi$ is defined as
\[\IW(\vphi)=\bigsqcup_{v\in V\Gamma} \SW(g;v),\] 
but with components containing only 2 vertices removed.

The ideal Whitehead graph is an invariant of the conjugacy class of the outer automorphism represented by $g$ and $\IW(\vphi^k)=\IW(\vphi)$ for each $k\in\ZZ_{>0}$ \cite{hm11, Thesis}.

The \textit{lamination train track (ltt) structure} $G(g)$ is obtained from its \emph{underlying graph} $\G$ by replacing each vertex $v \in V\G$ with $\LW(g; v)$: replace $v$ with a vertex for each directed edge at $v$ labeled with that direction. Then identify each of these new vertices with the corresponding vertex of $\LW(g; v)$. Vertices and edges of $\SW(g;v)$ are colored purple and the remaining vertices and (open) edges are colored red. Alternatively, one could start with $\bigsqcup_{v\in V\G} \LW(g; v)$, color $\LW(g; v)$ as just described, and then include a directed edge $[e,\ol{e}]$ for each directed edge $e\in E\G$. To simplify figures, if the local Whitehead graph at a vertex of $\G$ is complete we do not always draw it.

\begin{figure}[H]
\includegraphics[width=5.35in]{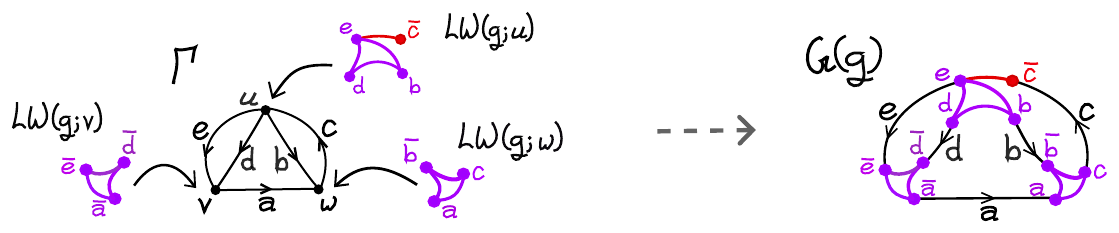}
\vspace*{4mm}
 \caption{The right-hand image is the ltt structure $G(\mathfrak{g})$ for the map $\mathfrak{g}$ of (\ref{g}) with the taken turns as in (\ref{turns}). The 3 local Whitehead graphs are colored, with the stable Whitehead graphs colored purple (missing only $\bar{c}$). The ideal Whitehead graph is the union of the purple graphs, i.e. of the stable Whitehead graphs.} \label{f:wg}
\end{figure}

%%%%%%%%%%%%%%%%%%%%%%%%%%%%%%%%%%%%%%%%%%%%%%%%%%%%%%%%%%%%%%%%%%%%
%%%%%%%%%%%%%%%%%%%%%%%%%%%%%%%%%%%%%%%%%%%%%%%%%%%%%%%%%%%%%%%%%%%%

\vspace*{-9mm}

\subsection{Full irreducibility criterion}{\label{ss:fic}}

We use the following criterion for proving that a tt map represents a fully irreducible outer automorphism.

\begin{prop}[{Full Irreducibility Criterion, \cite[Proposition 4.1]{IWGII}}]\label{prop:FIC}
	Suppose that $g\colon \G \to \G$ is a PNP-free, irreducible train track representative of $\vphi \in \out$ such that $M(g)$ is Perron--Frobenius and all the local Whitehead graphs are connected. Then $\vphi$ is fully irreducible.
\end{prop}

%%%%%%%%%%%%%%%%%%%%%%%%%%%%%%%%%%%%%%%%%%%%%%%%%%%%%%%%%%%%%%%%%%%%
%%%%%%%%%%%%%%%%%%%%%%%%%%%%%%%%%%%%%%%%%%%%%%%%%%%%%%%%%%%%%%%%%%%%

\subsection{Principal \& ageometric fully irreducible outer automorphisms}{\label{ss:principal}}

Principal fully irreducible outer automorphisms (sometimes just called principal outer automorphisms) were introduced in \cite{stablestrata} as analogues of pseudo-Anosov homeomorphisms with only $3$-pronged singularities: As in \cite{stablestrata}, we call a fully irreducible $\vphi \in \out$ \textit{principal} if $\IW(\vphi)$ is the disjoint union of $2r-3$ triangles.

We include now a short description of properties of principal fully irreducible elements omitting definitions not used elsewhere in this paper. The interested reader can find the definitions in \cite{loneaxes}, which also includes a nice explanation of the history at the start of \S 2.9. For a principal $\vphi \in \out$, the rotationless index is $i(\vphi)=\frac{3}{2}-r$. This implies $\vphi$ is ageometric and, by \cite[Theorem 4.5]{loneaxes}, all its tt representatives are stable. In particular, none of its tt representatives has a PNP (\cite[Theorem 3.2]{bf94}). By \cite[Theorem 4.7]{loneaxes} one also has that $\vphi$ has only a single axis (in the sense of \S \ref{ss:os}).

%%%%%%%%%%%%%%%%%%%%%%%%%%%%%%%%%%%%%%%%%%%%%%%%%%%%%%%%%%%%%%%%%%%%
%%%%%%%%%%%%%%%%%%%%%%%%%%%%%%%%%%%%%%%%%%%%%%%%%%%%%%%%%%%%%%%%%%%%

\subsection{Folds \& Stallings fold decompositions}{\label{ss:StallingsFoldDecompositions}}

~\vskip1pt

\parpic[r]{\includegraphics[width=1.7in]{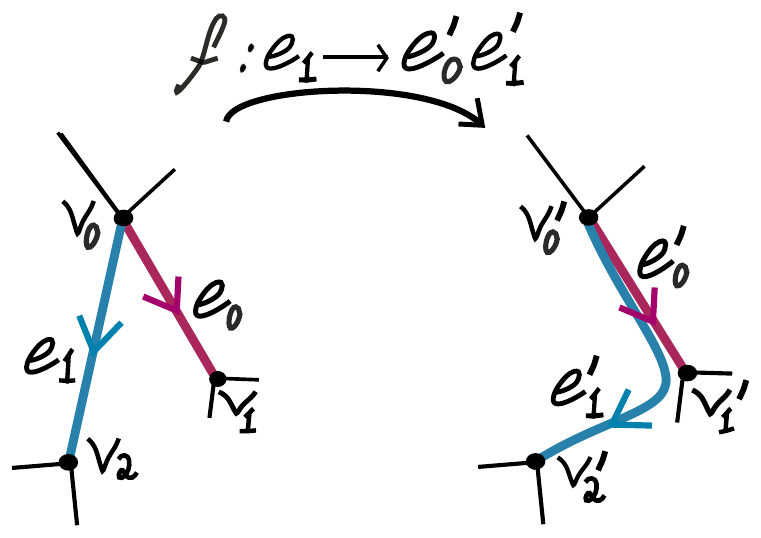}} 
Suppose $\G$ and $\G'$ are graphs viewed topologically and $e_0,e_1\in E^{\pm}\G$ are distinct edges emanating from a common vertex. We say $\G'$ is obtained from $\G$ by a \emph{proper full fold of $e_1$ over $e_0$} if there exist orientation-preserving homeomorphisms $\sigma_0\from [0,1]\to e_0$ and $\sigma_1\from [0,2]\to e_1$ so that 
$\G'=\G\backslash\sim$ is the topological quotient of $\G$ with respect to the equivalence relation $\sim$ defined by $\sigma_0(t)=\sigma_1(t)$ for each $t\in [0,1]$. We say that $\G'$ is obtained from $\G$ by a \emph{complete fold of $e_0$ and $e_1$} if instead $\sigma_1\from [0,1]\to e_1$ (identifying all of $e_0$ and $e_1$) and a \emph{partial fold of $e_0$ and $e_1$} if instead $\sigma_0\from [0,2]\to e_0$ (identifying initial segments of both). By a \emph{fold} (or sometimes \emph{Stallings fold}), we will mean either a proper full fold, complete fold, or partial fold. We call $\{e_0,e_1\}$ the \emph{folded turn} (or \emph{turn folded}). Observe that in a proper full fold the number of edges is preserved, in a complete fold it is reduced, and in a partial fold it is increased.

%\vspace*{-1.5mm}

Since they appear often, we describe here the notational conventions we use for proper full folds. 
Suppose that $f\from\G\to\G'$ is a single proper full fold of an edge $e_1$ over an edge $e_0$, as depicted above. Apart from $e_1$, each edge $e_k\in E\G$ is mapped to a single edge of $\G'$, which we call $e_k'$. The image of $e_1$ is an edge-path in $\G'$ consisting of 2 edges, the latter of which we call $e_1'$. The map $f$ is then defined by 

%\vspace*{-2mm}

$$
f =
\begin{cases}
e_1\mapsto e_0'e_1' \\
e_k \mapsto e_k' \text{ for } k\neq 1\\
\end{cases}
$$

\noindent and we just write $f\from e_1\mapsto e_0e_1$. We call the edge-labeling of $\G'$ just described the \emph{induced edge-labeling}. Note that it can be more convenient, especially where an orientation was chosen in advance, to write $f \from e_1 \mapsto e_1e_0$ for the map induced by folding $\bar{e}_1$ over $\bar{e}_0$. We may drop primes in complicated diagrams.

Stallings \cite{s83} showed that a surjective homotopy equivalence graph map $g \colon \G \to \G'$ factors as a composition of folds and a final homeomorphism, giving a \emph{Stallings fold decomposition}. 

\[
\xymatrix{\G=\Gamma_0 \ar[r]_{g_1} \ar@/^4pc/[rrrr]_{g=\mathfrak{g}_0} & \Gamma_1 \ar[r]_{g_2} \ar@/^3pc/[rrr]_{\mathfrak{g}_1} & \Gamma_2 \ar[r]_{g_3} \ar@/^2pc/[rr]_{\mathfrak{g}_2} & \dots \ar[r]_{g_n}  & \Gamma_n=\Gamma' \\}
\]

In a Stallings fold decomposition the folds $g_{k+1}$ (sometimes called \textit{Stallings folds}) of $e_0$ and $e_1$ in $\G_k$ additionally satisfy $\mathfrak{g}_k(\sigma_0(t))=\mathfrak{g}_k(\sigma_1(t))$ for each $t\in [0,1]$ and that the terminal vertices of $e_0$ and $e_1$ are distinct points in $\mathfrak{g}_k^{-1}(V\G')$. The latter property ensures $g$ is a homotopy equivalence.

In general, each fully irreducible $\vphi\in\out$ has many tt  representatives, each of which can have several distinct Stallings fold decompositions. 
By \cite[Theorem 4.7]{loneaxes}, principal outer automorphisms  (and their powers)  
each have a unique Stallings fold decomposition.

%%%%%%%%%%%%%%%%%%%%%%%%%%%%%%%%%%%%%%%%%%%%%%%%%%%%%%%%%%%%%%%%%%%%
%%%%%%%%%%%%%%%%%%%%%%%%%%%%%%%%%%%%%%%%%%%%%%%%%%%%%%%%%%%%%%%%%%%%

\subsection{The Outer space \texorpdfstring{$\os$}{CVn} \& its (principal) geodesics}{\label{ss:os}}

Culler--Vogtmann Outer space was first defined in \cite{cv86}. We refer the reader to \cite{FrancavigliaMartino,b15,v15} for background on Outer space, giving only an abbreviated discussion here. For $r\ge 2$ denote the (volume-1 normalized) Outer space for $F_r$ by $\os$.  Points of $\os$ are equivalence classes of volume-1 marked metric graphs $h:R_r\to\G$ where $R_r$ is the $r$-rose, and $\G$ is a finite volume-1 metric graph with betti number $b_1(\Gamma)=r$ and with all vertices of degree at least 3, and $h$ is a homotopy equivalence called a \emph{marking}. Outer space $\os$ has a simplicial complex structure with some faces missing: there is an open simplex for each marked graph (obtained by varying the lengths on the edges of that graph). The faces of a simplex are obtained by collapsing the edges of a forest (so that their lengths become zero). There is an asymmetric metric $d_{\cv}$ on $\os$.

There is an action (by isometries) of $\out$ on $\os$, given by changing the marking. We denote the quotient by $\mM_r$.

Given a Stallings fold decomposition of a tt map $g$, one can define a ``periodic fold line'' in $\os$ using Skora's \cite{s89} interpretation of the decomposition as a sequence of folds performed continuously.
In \cite[Lemma 2.27]{stablestrata} it is proved that periodic fold lines associated to tt maps are geodesics in the sense that,
given 3 points $\gamma(t_1),\gamma(t_2),\gamma(t_3)$ on the geodesic $\gamma$ with $t_1<t_2<t_3$,
we have $d_{\cv}(\gamma(t_1),\gamma(t_2))+d_{\cv}(\gamma(t_2),\gamma(t_3))= d_{\cv}(\gamma(t_1),\gamma(t_3))$. 

Since a general fully irreducible $\vphi \in \out$ can have several distinct Stallings fold decompositions, it can have several distinct periodic fold lines. However, the principal outer automorphisms that we are interested in will have just a single periodic fold line, its ``lone axis.''

%%%%%%%%%%%%%%%%%%%%%%%%%%%%%%%%%%%%%%%%%%%%%%%%%%%%%%%%%%%%%%%%%%%%
%%%%%%%%%%%%%%%%%%%%%%%%%%%%%%%%%%%%%%%%%%%%%%%%%%%%%%%%%%%%%%%%%%%%

\subsection{Fold-conjugate decompositions}{\label{ss:FoldConjugate}}

%\vskip1pt

Since an axis for a fully irreducible $\vphi\in\out$ has a periodic structure, it becomes useful to view its Stallings fold decompositions cyclically. With some work one can see that starting at a different fold in a decomposition now yields a tt map representing an element of $\out$ that is $\out$-conjugate to $\vphi$ and with the same axis. It is also possible to start the tt map ``in the middle of a fold.'' We formalize here these different notions of cyclically permuting a Stallings fold decomposition or, equivalently, shifting along an axis.

\parpic[r]{\includegraphics[width=.7in]{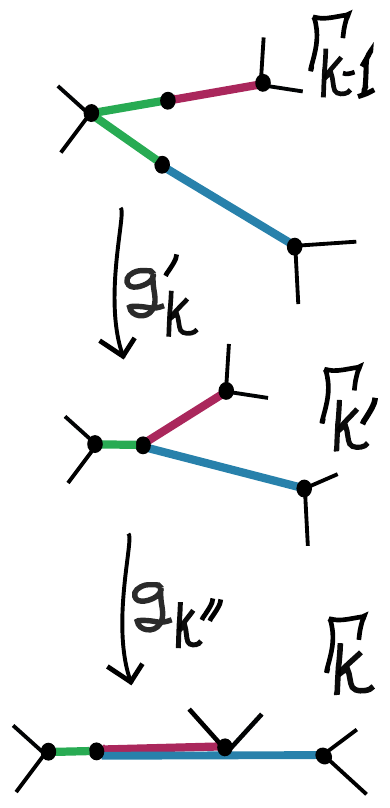}} 
A \emph{subdivided fold} will mean a fold written as a composition of a partial fold and then a completion of that fold, as depicted to the right. Suppose $\G_{0} \xrightarrow{g_1} \G_{1} \xrightarrow{g_2} \cdots \xrightarrow{g_{m-1}} \G_{m-1} \xrightarrow{g_m} \G_m$ and $\G'_{0} \xrightarrow{h_1} \G'_{1} \xrightarrow{h_2} \cdots \xrightarrow{h_{n-1}} \G'_{n-1} \xrightarrow{h_n} \G'_n$ are Stallings fold decompositions of homotopy equivalence tt maps $g \from \Gamma \to \Gamma$ and $h \from \Gamma' \to \Gamma'$, respectively. We say the decompositions are \emph{fold-conjugate} if one of the following three possibilities occurs. In all cases the the vertical arrows are label-preserving graph isomorphisms.

Either the folds of the 2nd sequence are a cyclic shift of those of the 1st: there is some $j$ so that diagram (\ref{a}) commutes.

\begin{equation}\label{a}
  \begin{tikzcd}
		\Gamma_j \ar{r}{g_{j+1}} \ar{d} & \Gamma_{j+1} \ar{r}{g_{j+2}} \ar{d} & \dots \ar{r}{g_m} & \Gamma_m = \Gamma_0 \ar{r}{g_1} \ar{d} &  \dots \ar{r}{g_{j-1}} & \Gamma_{j-1} \ar{r}{g_{j}} \ar{d} & \Gamma_j \ar{d}\\
		\Gamma'_0 \ar{r}{h_1} & \Gamma'_{1} \ar{r}{h_2} & \dots \ar{r}{h_{n-j}} & \Gamma'_{n-j} \ar{r}{h_{n-j+1}} &  \dots \ar{r}{h_{n-1}} & \Gamma'_{n-1} \ar{r}{h_{n}}  & \Gamma'_n 
\end{tikzcd}
\end{equation}

\noindent Or, the folds of the 1st are a cyclic shift of those of the 2nd, but starting midfold: there is some $\ell$ so that after subdividing $h_{\ell}$ into $\G'_{\ell-1} \xrightarrow{h_{\ell}'} \G'_{{\ell}'} \xrightarrow{h''_{\ell}} \G'_{\ell}$ diagram (\ref{b}) commutes.
\begin{equation}\label{b}
 \begin{tikzcd}
		\Gamma_0 \ar{r}{g_{1}} \ar{d} & \Gamma_{1} \ar{r}{g_{2}} \ar{d} & \dots \ar{r}{g_{m-{\ell}}} & \Gamma_{m-{\ell}} \ar{r}{g_{m-{\ell}+1}} \ar{d} &  \dots \ar{r}{g_{m-1}}  & \Gamma_{m-1} \ar{r}{g_{m}} \ar{d} & \Gamma_m \ar{d}\\
		\Gamma'_{\ell'} \ar{r}{h''_{\ell}} & \Gamma'_{\ell} \ar{r}{h_{\ell+1}} & \dots \ar{r}{h_{n}} & \Gamma'_{n} = \Gamma'_0 \ar{r}{h_{1}}  & \dots \ar{r}{h_{\ell-1}} & \Gamma'_{\ell-1} \ar{r}{h'_{\ell}}  & \Gamma'_{\ell'} 
\end{tikzcd}
\end{equation}

\noindent Or, the folds of the 2nd are a cyclic shift of those of the 1st, but starting midfold: there is some $k$ so that after a fold subdivision of some $g_k$ into $\G_{k-1} \xrightarrow{g_k'} \G_{k'} \xrightarrow{g''_{k}} \G_k$ diagram (\ref{c}) commutes.
\begin{equation}\label{c}
\begin{tikzcd}
		\Gamma_{k'} \ar{r}{g''_{k}} \ar{d} & \Gamma_{k} \ar{r}{g_{k+1}} \ar{d} & \dots \ar{r}{g_m} & \Gamma_m = \Gamma_0 \ar{r}{g_1} \ar{d}  & \dots \ar{r}{g_{k-1}} & \Gamma_{k-1} \ar{r}{g'_k} \ar{d} & \Gamma_{k'} \ar{d}\\
		\Gamma'_0 \ar{r}{h_1} & \Gamma'_{1} \ar{r}{h_2} & \dots \ar{r}{h_{n-k}} & \Gamma'_{n-k} \ar{r}{h_{n-k+1}} & \dots \ar{r}{h_{n-1}} & \Gamma'_{n-1} \ar{r}{h_{n}}  & \Gamma'_{n} 
\end{tikzcd}
\end{equation}

Fold-conjugate decompositions are \emph{partial-fold conjugate} if they fall in the case of (\ref{b}) or (\ref{c}), that is if one of the fold sequences is a cyclic shift of the other but for starting midfold. They are called partial-fold conjugate because the conjugating sequence contains a partial fold. In this case, $\mid n-m \mid =1$.

Fold-conjugate tt maps represent $\out$-conjugate outer automorphisms, hence share all conjugacy class invariants, such as ideal Whitehead graphs, and whether or not a map is fully irreducible, and then also principal.

We can also consider an equivalence relation generated by (\ref{a})-(\ref{c}) on a more general class of fold sequences. This class of fold sequences would include pathological examples with many ``unnecessary'' subdivisions of one or more folds, as well as allowing us to start in the middle of any fold in the sequence. The benefit of such a relation is that it would allow one to define and understand an $\out$-conjugacy class invariant that is the minimal number of folds in an equivalence class with respect to the partial-fold conjugacy relation.

\vskip10pt

%%%%%%%%%%%%%%%%%%%%%%%%%%%%%%%%%%%%%%%%%%%%%%%%%%%%%%%%%%%%%%%%%%%%
%%%%%%%%%%%%%%%%%%%%%%%%%%%%%%%%%%%%%%%%%%%%%%%%%%%%%%%%%%%%%%%%%%%%

\section{The Rank-3 Principal Stratum Automaton \texorpdfstring{$\widehat{\mA_3}$}{\^A3}}
	\label{ss:automaton}

The ``Lonely Direction Automaton'' of \cite{automaton} includes as directed loops only the Stallings fold decompositions of rotationless principal fully irreducible $\vphi\in\outt$. To include directed loops for all principal fully irreducible outer automorphisms, instead of just their rotationless powers, we construct a new directed graph we call the ``Rank-3 Principal Stratum Automaton'' $\widehat{\mA_3}$. Before introducing $\widehat{\mA_3}$, we review $\mA_3$ as defined and constructed in \cite{automaton}.

\subsection{The Lonely Direction Automaton $\mA_3$}
	\label{ss:A3}

In \cite{automaton}, $\mA_3$ is defined in a constructive manner using ``comprehensive loops'' and ``permissible folds.'' We give a simplified description of the construction now. $\mA_3$ is the disjoint union of the strongly connected components of the following directed graph. 

{\bf{Nodes:}} There is a node for each order pair $(\Gamma, \ell)$ such that \\
\indent $\bullet$ $\G$ is a rank-3 graph (in the sense of \S \ref{s:background}) satisfying that all vertices are of valence-3, except a single vertex $v$ of valence-4 and\\
\indent $\bullet$ $\ell$ is a loop in $\G$ traversing each edge of $\G$ in at least one of the two directions and\\
\indent $\bullet$ there is a direction $d$ at $v$ so that $\ell$ takes each turn of $\G$ except precisely two of the turns at $v$ containing $d$.

The node is the ltt structure $ltt(\Gamma, \ell)$ defined in the same manner as in \S \ref{ss:WGs} except that the colored edges now come from the turns taken by $\ell$ and the single red vertex is that labeled by $d$. The nodes for $ltt(\Gamma, \ell)$ and $ltt(\Gamma', \ell')$ are equivalent (the same node), if there exists a decoration-preserving (meaning preserving both that color and labeling) graph isomorphism from $ltt(\Gamma, \ell)$ to $ltt(\Gamma', \ell')$. That is, the only recorded information from $\ell$ is the turns it takes.

{\bf{Directed edges:}} There is a directed edge $E(f)$ from $ltt(\Gamma, \ell)$ to $ltt(\Gamma', \ell')$ precisely when there is a proper full fold $f\from\G\to\G'$ that does not fold any turn taken by $\ell$. That is, the turn folded by $f$ is not represented by a colored edge in $ltt(\Gamma, \ell)$.

%And each $\mA(\Gamma, \ell)$ is then defined as follows. 
% The construction starts with an ordered pair $(\Gamma, \ell)$ where $\G$ is a graph in the sense of \S \ref{s:background} and $\ell$ is a loop traversing each edge of $\G$ (in at least one of the two orientations). 

\subsection{The Rank-3 Principal Stratum Automaton $\widehat{\mA_3}$}
	\label{ss:A3hat}  

We build on $\mA_3$ to construct $\widehat{\mA_3}$ by adding in ``permutation arrows'' where ltt structure permutation relabelings exist in the following sense:

\begin{df}[Graph/permutation relabeling]
	\label{d:relabeling}
Suppose that $\G$ is a directed labeled graph and $\sigma$ is a permutation of $E^{\pm}\G$ so that $\sigma(\bar{e})=\overline{\sigma({e})}$ for each $e\in E^{\pm}\G$. By $\sigma\cdot\G$ we mean the directed labeled graph obtained from $\G$ by replacing each edge-label $e$ with $\sigma(e)$. By a \textit{graph relabeling (by $\sigma$)} we mean the graph isomorphism $g_{\sigma}\from\G\to\sigma\cdot\G$ defined by that for each $e\in E\G$ the image of $e$ is the directed edge $\sigma(e)\in E(\sigma\cdot\G)$. Suppose that $G$ is an ltt structure with underlying graph $\G$, we let $\sigma\cdot G$ denote the ltt structure on $\sigma\cdot\G$ obtained from $G$ by relabeling the black edges and vertices via $\sigma$. We then consider $\sigma\cdot G$ to be a \textit{permutation relabeling} of $G$ by $\sigma$. 

Recall from \S \ref{ss:StallingsFoldDecompositions} that for a proper full fold $f\from \G\to\G'$ we know that $| E\G | = | E\G' |$ and $\G'$ comes with an induced edge-labeling. 
%The induced edge-labeling then gives an identification of $E\G$ with $E\G'$ so that we can also talk about a map $g_{\sigma}\from\G\to\sigma\cdot\G'$ that is really a composition of a proper full fold $f$ and a permutation relabeling map. By an abuse of notation, we may refer to both the maps $\G'\to\sigma\cdot\G'$ and $\G\to\sigma\cdot\G'$ by $g_{\sigma}$ when we believe that the proper full fold is clear from context. \naomi{MAYBE: 
The induced edge-labeling then gives an identification of $E\G$ with $E\G'$, and with this induced edge-labeling we can ``push forward'' the graph relabeling to $g_\sigma': \Gamma' \to \sigma \cdot \Gamma'$. Given a fold $f$ of $\Gamma$ and a graph relabeling $g_\sigma$, we obtain a fold of $\sigma \cdot \Gamma$ by ``conjugating'': $f' := g_\sigma' \circ f \circ g_\sigma^{-1}$. In what follows we will drop the prime from the notation $g_\sigma'$, letting $g_\sigma$ refer to the permutation relabeling after pushing forward as many times as necessary in each context.
%}
\end{df}

\bigskip

The following lemma provides a concrete description of the $\mA_r$, and $\widehat{\mA_r}$, nodes.

\smallskip

\begin{lem}[Lonely Direction Lemma]\label{lemma:LonelyDirection} Suppose $\vphi\in\out$ is a principal fully irreducible outer automorphism, then each tt representative of $\vphi$ is partial-fold conjugate to a tt map $g\from \G\to\G$ satisfying:
\begin{enumerate}[(a)]
\item each $SW(v,g)$ is a triangle and there are $2r-3$ such triangles, and
\item all vertices of $\G$ have valence 3 except a single vertex which has valence 4, and
\item all but one direction in $\G$ is $g$-periodic and this direction is at the valence-4 vertex of $\G$, and
\item the nonperiodic direction is contained in precisely 1 turn taken by $g$.
\end{enumerate}
\end{lem}

\begin{proof}

Suppose $h\from \G'\to\G'$ is a tt representative of $\vphi$ and let $R$ denote the rotationless power of $\vphi$. Then $h^R\from \G'\to\G'$ is a tt representative of $\vphi^R$, and each $h$-periodic vertex and direction is fixed by $h^R$. By the proof of \cite[Corollary 3.8]{loneaxes}, $h^R$, and hence also $h$, has at most 1 nonperiodic vertex and that vertex has precisely 2 gates. As in the proof of \cite[Corollary 3.8]{loneaxes}, by performing a fold at that vertex, one obtains a tt representative $g'$ of $\vphi^R$ for which all vertices are fixed and have at least $3$ fixed directions. Since $\vphi$ and $\vphi^R$ are lone axis fully irreducible outer automorphisms, this fold is within their shared axis $\mA$.

Because $\mA$ is both $h$- and $h^R$-periodic, with both periods starting and ending at (re-marked copies of) $\G'$, the \cite[Corollary 3.8]{loneaxes} fold, let us call it $f$, is a fold of $\G'$ occurring at the start of the Stallings fold decompositions of both $h$ and $h^R$. If $f$ is part of a subdivided fold (in the sense of \S \ref{ss:FoldConjugate}), since it occurs within the axis, there is some full fold $F$ in $\mA$ and partial fold $f'$ in $\mA$ so that $F=f\circ f'$. When you do the folds cyclically, since the decomposition of $h$ lies on $\mA$, the copy of $f'$ at the end of the decomposition of $h$ must compose with $f$ at the beginning of the next iteration of $h$ to give $F$. Thus, $f'$ would appear at the end of the Stallings fold decompositions of $h$ and $h^R$. Hence, in fact, letting $g$ denote the conjugation of $h$ by the fold, i.e. $f \circ h= g\circ f$, we have $g'=g^R$.

We now show that $g$ is a tt representative of $\vphi$. Since $g'=g^R$ is a tt map, no power of $g$ can cause backtracking on an element of $E\G$. Since quotienting by a partial fold cannot change the induced map of fundamental groups, we are left to show that $g$ maps vertices to vertices. But each vertex has at least 3 $g'$-gates, so cannot be mapped within its Stallings fold decomposition to a nonvertex point.
%\begin{enumerate}[(a)]
%\item 

Since $\vphi\in\out$ is principal, its ideal Whitehead graph $\IW(\vphi)$, hence also $\IW(\vphi^k)$ for each $k\in\ZZ_{>0}$, is the disjoint union of $2r-3$ triangles. By \cite[Lemma 4.5]{loneaxes} and since $\IW(\vphi)=\IW(\vphi^k)$ has no cut vertices for each $k\in\ZZ_{>0}$, no tt representative $\tau$ of any $\vphi^k$ has a PNP. Thus, since $IW(\vphi^k)$ is the disjoint union of the $SW(\tau,v)$ having at least 3 vertices, for any such $\tau$ and vertex $v$ with at least $3$ periodic directions, $SW(\tau,v)$ is a triangle. Since each vertex of $g^R$ has at least $3$ fixed directions, each vertex of $g$ has at least $3$ periodic directions. Hence, each $SW(g,v)$ is a triangle and there are $2r-3$ such triangles, proving (a).

We now prove (b)-(c). Since each $SW(g,v)$ is a triangle, there are precisely 3 periodic directions (hence gates) at each vertex $v\in V\G$. By \cite[Lemma 3.6]{loneaxes}, $g^R$, hence also $g$, has precisely one illegal turn. Thus, $\G$ has precisely one $g$-nonperiodic direction. And we have that $\G$ has a unique vertex of valence $4$, and all other vertices have valence $3$.

Part (d) follows from \cite[Lemma 2]{automaton}, after observing that each periodic direction must be contained in 2 turns (by the structure of the stable Whitehead graphs) so that the turn described in \cite[Lemma 2]{automaton} must contain the unique nonperiodic direction. \qedhere

\end{proof}

\bigskip

We now define the rank-3 principal stratum automaton $\widehat{\mA_3}$:
  
\begin{df}[Rank-3 Principal Stratum Automaton $\widehat{\mA_3}$]
The \emph{Rank-3 Principal Stratum Automaton} $\widehat{\mA_3}$ is obtained from the ``Lonely Direction Automaton'' of \cite{automaton} by adding a bi-directed edge labeled with $\sigma$ for each ltt structure permutation relabeling $G\to\sigma\cdot G$, which implicitly records the directed edge $\sigma$, labeling the edge $[G,\sigma\cdot G]$, and $\sigma^{-1}$, labeling the edge $[\sigma\cdot G, G]$. For each $r\geq 3$, the \emph{Rank-r Principal Stratum Automaton} $\widehat{\mA_r}$ is defined analogously.
\end{df} 

Note that, once a choice of labels has been made for the black edges of a single node, this determines the labeling of the black edges in each other node, as well as the labels on the directed edges. In order to explicitly write out $\widehat{\mA_3}$, we make such a choice and call the decoration-preserving graph isomorphism class of $\widehat{\mA_3}$, endowed with such a choice of labeling, a \emph{graph-labeling class}.

\bigskip

Figure \ref{f:mA3} depicts $\widehat{\mA_3}$ with the permutation edges depicted in green. %While we refer the reader to \cite{automaton} for the details of its construction beyond what is given above, one can note that it is the maximal strongly connected components (in rank 3 there is only one such component containing a loop, up to permutation relabeling) of a graph where \begin{itemize}
%	\item vertices are ltt structures of the appropriate rank satisfying the ``Lonely Direction Property'' dictated by Lemma \ref{lemma:LonelyDirection} and
%	\item directed edges are either folds compatible with the ltt structures at the endpoints, or permutation relabelings as discussed in our definition. 
%\end{itemize}
The nodes and black directed edges are as described in \S\ref{ss:A3}. (Strictly speaking, what is shown is a maximal strongly connected component of $\widehat{\mA_3}$; in rank 3 there is only one such component containing a loop, up to permutation relabeling.)
Following a directed loop thus gives a sequence of folds and permutations which could be a decomposition of a principal fully irreducible outer automorphism, though one can easily produce directed loops that do not give such elements. For example, the folds represented by the directed edges of the loop may be supported in a proper subgroup of the fundamental group (so that the map will be reducible).

\subsection{Loops in $\widehat{\mA_3}$}
	\label{ss:A3hatLoops}  

Consider a directed loop in the automaton with the edge labels reading as follows (the $f_k$ are folds and the $\sigma_k$ are graph-relabeling permutations):

\vspace*{-5mm}

\begin{equation}\label{decomp}
\G_{0}' \xrightarrow{g_{\sigma_0}} \G_{0} \xrightarrow{f_1} \G_{1}' \xrightarrow{g_{\sigma_1}} \G_{1} \xrightarrow{f_2} \G_{2}' \xrightarrow{g_{\sigma_2}} \G_{2} \xrightarrow{f_3} \cdots \xrightarrow{g_{\sigma_{n-1}}} \G_{n-1} \xrightarrow{f_n} \G_{n}' \xrightarrow{g_{\sigma_{n}}} \G_{n}=\G.
\end{equation}

%\naomi{MAYBE: 
Fixing a labeling on $\Gamma_0'$, we induce labelings on all subsequent $\Gamma_i$ and $\Gamma_i'$. Using the notation described in Definition~\ref{d:relabeling} we can then define:
\begin{equation}\label{gammaj}
\Gamma_k^{(j)}:=(\sigma_k\circ\cdots\circ\sigma_{k-j})^{-1}\cdot\Gamma_k\quad \text{and}
\end{equation}

\begin{equation}\label{fj}
f_k^{(j)}:= 
g_{\sigma_{k-j-1}\circ\cdots\circ\sigma_0}^{-1} 
\circ f_k \circ 
g_{\sigma_{k-j-1}\circ\cdots\circ\sigma_0}
%(\sigma_0^{-1}\circ\cdots\circ\sigma_{k-j-1}^{-1}) 
\from \Gamma_{k-1}^{(j)} \to \Gamma_{k}^{(j+1)}.
\end{equation}

\smallskip

\noindent We have the following commutative diagram indicating how the map of (\ref{decomp}) can be rewritten as folds and then possibly a single ltt structure permutation relabeling. 
\begin{equation}\label{commutative1}
\begin{tikzcd}
		&&&&&& \Gamma_n=\Gamma_0 
\\
		&&&&&\Gamma_{n-1} \ar{r}{f_n} & \Gamma_n' \ar{u}{g_{\sigma_n}} \\
		%&&&&& \ar[dashed]{d} & \ar[dashed]{d} \\
&&&&&&\\
		&&&	\Gamma_2 \ar[dashed]{rr} && \Gamma_{n-1}^{(n-3)} \ar[dashed]{uu} \ar{r}{f_n^{(n-3)}} & \Gamma_n^{(n-2)} \ar[dashed]{uu} \\
		&&	\Gamma_1 \ar{r}{f_2} & \Gamma_2' \ar{u}{g_{\sigma_2}} \ar[dashed]{rr} && \Gamma_{n-1}^{(n-2)} \ar{u}{g_{\sigma_2}} \ar{r}{f_n^{(n-2)}} & \Gamma_n^{(n-1)} \ar{u}{g_{\sigma_2}} \\
		&	\Gamma_0 \ar{r}{f_1} & \Gamma_1' \ar{u}{g_{\sigma_1}} \ar{r}{f_2'} & \Gamma_2'' \ar{u}{g_{\sigma_1}} \ar[dashed]{rr} && \Gamma_{n-1}^{(n-1)} \ar{u}{g_{\sigma_1}} \ar{r}{f_n^{(n-1)}} & \Gamma_n^{(n)}  \ar{u}{g_{\sigma_1}} \\
		 & \Gamma_0' \ar{u}{g_{\sigma_0}} \ar{r}{f_1'} & \Gamma_1'' \ar{u}{g_{\sigma_0}} \ar{r}{f_2''} & \Gamma_2''' \ar{u}{g_{\sigma_0}} \ar[dashed]{rr} && \Gamma_{n-1}^{(n)} \ar{u}{g_{\sigma_0}} \ar{r}{f_n^{(n)}} & \Gamma_n^{(n+1)} \ar{u}{g_{\sigma_0}}
	\end{tikzcd}
\end{equation}

\bigskip

We strengthen here \cite[Proposition 6]{automaton}:

\setcounter{thmx}{2}
 \begin{thmx}\label{p:PrincipalGivesLoop}
Suppose $g$ is a train track representative of a principal fully irreducible $\vphi\in\outt$.
Then the Stallings fold decomposition of $g$ is partial-fold conjugate to one determining a directed loop in $\widehat{\mA_3}$.
\end{thmx}

\begin{proof}
Suppose $g$ is a tt representative of a principal fully irreducible $\vphi\in\outt$. By \cite[Proposition 6]{automaton}, some power $g^p$ of $g$ is fold-conjugate (really partial-fold conjugate) to a map given by a directed loop in $\mA_3$, and we explain how to fine-tune the proof of \cite[Proposition 6]{automaton} for this level of precision that we need. Since tt representatives of the same principal (hence lone axis) outer automorphism have Stallings fold decompositions yielding the same axis, they are in fact fold-conjugate through folds of the axis. Hence, if the folds for one of these tt representatives is represented by a loop in ${\mA_3}$, then so is the other, provided they both start and end on the kinds of ltt structures represented in the automata, i.e those whose underlying graph satisfies the ``Lonely Direction property'' of \cite{automaton}. This is addressed in Lemma \ref{lemma:LonelyDirection}.

Since $\vphi$ is principal, hence lone axis, $g$ has a unique Stallings fold decomposition 
\begin{equation}\label{PrincipalDecomp}
\G=\G_{0} \xrightarrow{g_1} \G_{1} \xrightarrow{g_2} \G_2 \xrightarrow{g_3} \cdots \xrightarrow{g_{t-2}} \G_{t-1} \xrightarrow{g_{t-1}} \G_t \xrightarrow{g_{\sigma}} \G_{0}=\G.
\end{equation}

\noindent If necessary, we can increase $p$ so $\sigma^p$ is the identity. Since $g^p$ can be decomposed as
\begin{equation}\label{PrincipalDecompPower}
(\G_{0} \xrightarrow{g_1} \G_{1} \xrightarrow{g_2} \G_2 \xrightarrow{g_3} \cdots \xrightarrow{g_{t-2}} \G_{t-1} \xrightarrow{g_{t-1}} \G_t \xrightarrow{g_{\sigma}}\G_{0})^p,
\end{equation}
\noindent in light of (\ref{commutative1}), we obtain a Stallings fold decomposition of $g^p$ as
\begin{multline}\label{PrincipalDecompPowerClean}
\G=\G_{0} \xrightarrow{g_1} \G_{1} \xrightarrow{g_2} \cdots \xrightarrow{g_{t-2}} \G_{t-1} \xrightarrow{g_{t-1}} \G_t=\G_{0}' \xrightarrow{g_1'} \G_{1}' \xrightarrow{g_2'} \cdots \xrightarrow{g_{t-2}'} \G_{t-1}' \xrightarrow{g_{t-1}'} \G_t'=\G_{0}'' \cdots\\
\cdots\xrightarrow{g_{t-1}^{(p-2)}} \G_t^{(p-2)}=\G_{0}^{(p-1)} \xrightarrow{g_1^{(p-1)}} \G_{1}^{(p-1)} \xrightarrow{g_2^{(p-1)}} \cdots \xrightarrow{g_{t-2}^{(p-1)}} \G_{t-1}^{(p-1)} \xrightarrow{g_{t-1}^{(p-1)}} \G_t^{(p-1)}\xrightarrow{g_{\sigma}^p=Id} \G_{0}=\G.
\end{multline}

\noindent Since $g^p$ must also be principal, hence lone axis, this is the unique Stallings fold decomposition for $g^p$, so must be partial-fold conjugate to a fold decomposition represented by a directed loop in $\mA_3$. 

If $\G_{0}$ has a valence-4 vertex then, by the proof of  \cite[Proposition 6]{automaton} and the first paragraph of this proof, $g^p$ itself is represented by a directed loop in $\mA_3$. Now (\ref{PrincipalDecomp}) is, but for the addition of $g_{\sigma}$, a subpath of the directed loop  (\ref{PrincipalDecompPowerClean}) in $\mA_3$. And the Rank-3 Principal Stratum Automaton $\widehat{\mA_3}$ is obtained from $\mA_3$ by adding in edges for all possible ltt structure permutation relabelings, closing up the path via $g_{\sigma}$ to a loop. So $g$ will be represented by a directed loop in $\widehat{\mA_r}$. 
   
Now suppose that $\G_{0}$ is fully trivalent. As in (\ref{commutative1}), 
$$\G_{0} \xrightarrow{g_1} \G_{1} \xrightarrow{g_2} \G_2 \xrightarrow{g_3} \cdots \xrightarrow{g_{t-2}} \G_{t-1} \xrightarrow{g_{\sigma} }\sigma\cdot\G_{t-1} \xrightarrow{g_{\sigma}\circ g_{t-1}\circ g_{\sigma}^{-1}} \sigma\cdot\G_t=\G_{0}$$ 
composes to be the same map as (\ref{PrincipalDecomp}). We can see that $\G_1$ has a valence-4 vertex as follows. Complete the fold of the illegal turn at $\G_0$. If it fully identifies two edges 
\vspace*{-4mm}
\parpic[r]{\includegraphics[width=1.5in]{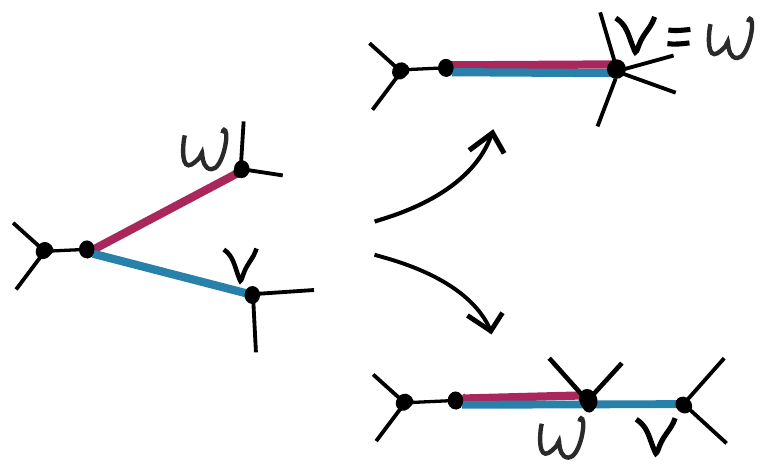}} 
\noindent (upper right-hand image), the terminal vertex would have valence-5 and neither $g$ nor its fold-conjugates could have represented a principal $\vphi\in \outt$, a contradiction. So $g_1$ must be a proper full fold (lower right-hand image) and $\G_1$ now has a valence-4 vertex at the image of $w$. So the previous paragraph indicates that $\widehat{\mA_3}$ contains a directed loop representing 

$$ \G_{1} \xrightarrow{g_2} \G_2 \xrightarrow{g_3} \cdots \xrightarrow{g_{t-2}} \G_{t-1} \xrightarrow{g_{\sigma}} \sigma\cdot\G_{t-1} \xrightarrow{g_{\sigma}\circ g_{t-1}\circ g_{\sigma}^{-1}} \sigma\cdot\G_t=\G_{0} \xrightarrow{g_1} \G_{1}.$$ 
So in this case also the Stallings fold decomposition of $g$ is partial-fold conjugate to one determining a directed loop in $\widehat{\mA_3}$. \qedhere

\end{proof}

\bigskip

\begin{prop}\label{lemma:map3}
	Each principal fully irreducible $\vphi \in \outt$ contains in its Stallings fold decomposition the left-hand graph. Thus, each principal axis must pass through the simplex with underlying graph both the second and third graphs in the image.

\begin{center}
\includegraphics[height=.4in]{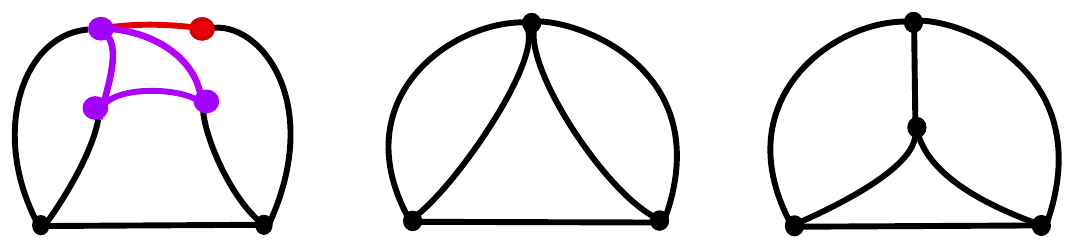}
\end{center}
\end{prop}

%\smallskip

\begin{proof} 
By Theorem \ref{p:PrincipalGivesLoop}, the Stallings fold decomposition of any representative $g$ of  a principal fully irreducible $\vphi \in\outt$ is partial-fold conjugate to one determining a directed loop in $\widehat{\mA_3}$.

\parpic[r]{\includegraphics[width=2.75in]{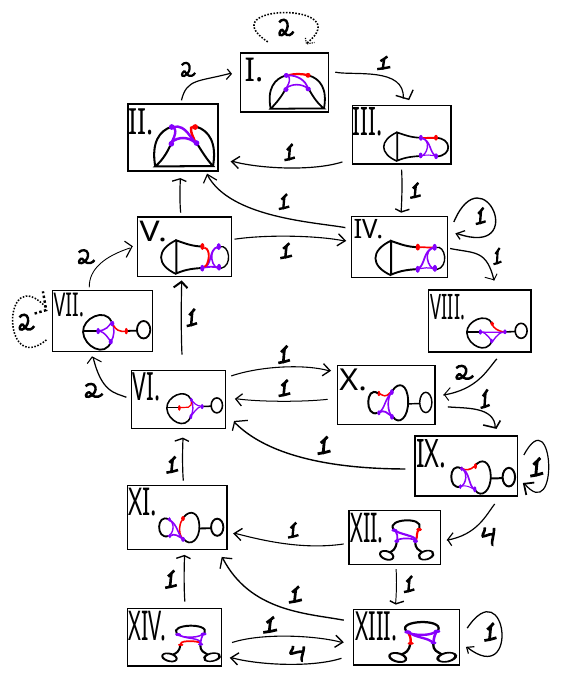}}
The image to the right schematically depicts $\widehat{\mA_3}$, with folds changing the labeling on the ltt structure included as dashed. In this figure we group together the nodes differing only by an edge-label permutation into a single node, indexed by roman numerals; this grouping is also shown in Figure~\ref{f:mA3}. %This figure, as well as Figure \ref{f:mA3}, have an abuse of notation where a collection of nodes differing only by an edge-label permutation are sloppily referred to as a single node (with a roman numeral assigned to it). 
Numbers on the arrows indicate how many distinct maps connect the isomorphism classes of ltt structures. It is interesting to note that, as drawn here and in Figure \ref{f:mA3}, $\widehat{\mA_3}$ has a clockwise orientation. 

One can then observe that, once Node I is removed, the maximal strongly connected component, which we call here $\mM_I$, includes none of Node I, Node II, or Node III. Further, in the edge-labeling scheme of Figure \ref{f:mA3}, no permutation within $\mM_I$ contains $c$ and no fold within $\mM_I$ maps $c$ over another edge. 

Hence, no directed loop within $\mM_I$ can define an irreducible map. And we thus have that any irreducible map represented by a directed loop in the automaton would need to contain Node I, which is the left-hand image in the statement of the proposition. The underlying graph of this ltt structure is the middle image. Inspection of all 4 folds entering Node I (including the self-maps), indicates that the axis must also pass through the third graph in the image.
\end{proof}

\vskip10pt

%%%%%%%%%%%%%%%%%%%%%%%%%%%%%%%%%%%%%%%%%%%%%%%%%%%%%%%%%%%%%%%%%%%%
%%%%%%%%%%%%%%%%%%%%%%%%%%%%%%%%%%%%%%%%%%%%%%%%%%%%%%%%%%%%%%%%%%%%

\section{The single-fold map}
\label{s:ourphi}

\noindent Consider the graph map $\mathfrak{g}$ defined by:

\vspace*{-4mm}

\begin{equation}\label{g}
\mathfrak{g} =
\begin{cases}
a \mapsto \overline{b} \\
b \mapsto \overline{d} \\
c \mapsto e \\
d \mapsto \overline{e}~\overline{c} \\
e \mapsto a\\
\end{cases}
\end{equation}

\vspace*{2mm}

\noindent In $\widehat{\mA_3}$, one can find $\mathfrak{g}$ represented by the directed loop:

%\vspace*{-2mm}

 \begin{center}
\begin{figure}[H]
\includegraphics[height=1.1in]{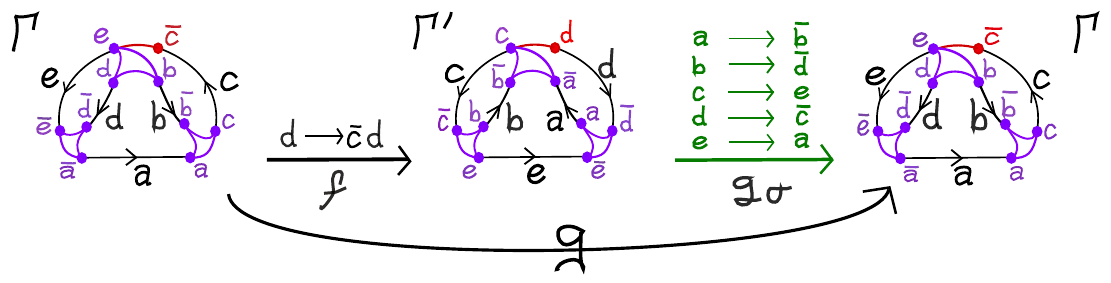}
\vspace*{4mm}
 \caption{The Stallings fold decomposition of $\mathfrak{g}$ is a fold $f$, then homeomorphism $g_{\sigma}$. This map is indicated in dotted gold in \ref{f:mA3}.} \label{f:map3}
\end{figure}
\end{center}

\vspace*{-4mm}

\begin{lem}\label{prop:map3}
	The map $\mathfrak{g}$ of Figure \ref{f:map3} represents a principal fully irreducible $\psi \in \outt$. \\
\end{lem}

\begin{proof} We first use the criterion of Proposition \ref{prop:FIC} to prove that $\mathfrak{g}$ represents an ageometric fully irreducible outer automorphism $\psi$. 

The direction map $D\mathfrak{g}$ is given by:\\

\[
\begin{tikzcd}
	\overline{c} \ar{r} & \overline{e} \ar{r} & \overline{a} \ar{r} & b \ar{r} & \overline{d} \ar{r} & c \ar{r} & e \ar{r} & a \ar{r} & \overline{b} \ar{r} & d \arrow[overlay, to path={-- ([yshift=2.5ex]\tikztostart.north) -| (\tikztotarget)}, rounded corners = 5pt]{llllllll}
\end{tikzcd}\]

\vspace*{5mm}

\noindent So one can check as in \cite[Lemma 5]{automaton} that the turns taken by $\{\mathfrak{g}^p(e)\mid p\in\ZZ_{>0}, e\in E\G\}$ are:

\begin{equation}\label{turns}
\{e,\overline{c}\},~\{a,\overline{e}\},~\{\overline{b},\overline{a}\},~\{d,b\},~\{\overline{e},\overline{d}\},
~\{\overline{a},c\},~\{b,e\},~\{\overline{d},a\},~\{c,\overline{b}\},~\{e,d\}.
\end{equation}

\vspace*{5mm}

\noindent Since the only illegal turn $\{d,\overline{c}\}$ is not taken, $\mathfrak{g}$ is a tt map. Since an adequately high power of the transition map is positive, $\mathfrak{g}$ is Perron--Frobenius. Alternatively, one can note that $\mathfrak{g}^{13}(d)$ contains all edges of $\Gamma$ and $d$ is in $\mathfrak{g}^{2}(a)$, $\mathfrak{g}(b)$, $\mathfrak{g}^{4}(c)$, and $\mathfrak{g}^{3}(e)$. Thus all edges of $\Gamma$ are in $\mathfrak{g}^{15}(a)$, $\mathfrak{g}^{16}(b)$, $\mathfrak{g}^{17}(c)$, and $\mathfrak{g}^{16}(e)$. So $\mathfrak{g}^{17}$ maps every edge over every edge.

In light of (\ref{turns}), the local Whitehead graphs are triangles at each valence-3 vertex, and at the valence-4 vertex they are the union of the Figure \ref{f:map3} colored edges. Hence, each local Whitehead graph is connected. Using the sage package of Coulbois \cite{c12}, for example, one can check that there are indeed no PNPs.

By Proposition \ref{prop:FIC}, $\mathfrak{g}$ thus represents an ageometric fully irreducible outer automorphism $\psi$. Further, since all directions but $\overline{c}$ are periodic, the ideal Whitehead graph would be a union of 3 triangles. Hence, $\vphi$ is in fact principal.
\end{proof}

\vskip10pt

%%%%%%%%%%%%%%%%%%%%%%%%%%%%%%%%%%%%%%%%%%%%%%%%%%%%%%%%%%%%%%%%%%%%
%%%%%%%%%%%%%%%%%%%%%%%%%%%%%%%%%%%%%%%%%%%%%%%%%%%%%%%%%%%%%%%%%%%%

\section{Perron numbers \& Minimal stretch factors}
\label{s:perron}

\smallskip

A \textit{weak Perron number} is an algebraic integer $\lambda$ that is at least $|\alpha|$ for each conjugate $\alpha$ of $\lambda$.  A weak Perron number where the inequality is strict is called a \textit{Perron number}. Perron numbers are precisely the spectral radii of nonnegative aperiodic integral matrices.

We will need:

\begin{lem}\label{lemma:Perron} For each $r\geq 2$, the stretch factor of each fully irreducible $\vphi\in\out$ is a Perron number.
\end{lem}

\begin{proof}
Suppose that $\vphi\in\out$ is a fully irreducible outer automorphism. Its stretch factor is the PF eigenvalue $\lambda(g)$ of the transition matrix $M(g)$ for any expanding irreducible tt map $g$. By definition, $M(g)$ is a nonnegative integer matrix. By \cite[Lemma 2.4(2)]{k14}, we know $M(g)$ is primitive. Hence, $M(g)$ is aperiodic. Since Perron numbers are precisely the spectral radii of nonnegative aperiodic integral matrices, $\lambda(g)$ is a Perron number, as desired.
\end{proof}

\bigskip

\setcounter{thmx}{0}
\begin{thmx}\label{t:minimal}
 The stretch factor of $\psi$ is minimal among fully irreducible elements in $\outt$. 
\end{thmx}

\begin{proof}
The transition matrix for $\mathfrak{g}$ is:
\vspace{-0.2cm}
\begin{equation*}
M(\mathfrak{g}) = 
\begin{bmatrix}
0 & 0 & 0 & 0 & 1  \\
1 & 0 & 0 & 0 & 0  \\
0 & 0 & 0 & 1 & 0  \\
0 & 1 & 0 & 0 & 0  \\
0 & 0 & 1 & 1 & 0  
\end{bmatrix}.
\end{equation*}

\noindent Thus the stretch factor of $\mathfrak{g}$ is the largest real root of the characteristic polynomial of $M(\mathfrak{g})$, which is $q(x)=x^5-x-1$.

Graphs in $CV_3$ have between 3 and 6 edges. Suppose $\vphi\in\outt$ has a tt representative on a graph with 6 edges. By an Euler characteristic argument, each graph with 6 edges is trivalent. If $\vphi$ is fully irreducible, then its Stallings fold decomposition must consist of at least one fold. As in the proof of Theorem \ref{p:PrincipalGivesLoop} (and its corresponding image), the graph at the completion of the fold will either have a valence-4 vertex (if it is a proper full fold) or a valence-5 vertex (if it completely identifies 2 edges). In either case an Euler characteristic argument indicates that the number of edges would have decreased from 6. Now, fold-conjugate outer automorphisms are in the same conjugacy class, so have the same stretch factor. Thus, if, a fully irreducible stretch factor is achieved by a tt map on a 6-edge graph, it is also achieved by a train track map on a graph with fewer edges. In fact, if a given stretch factor is achieved by a principal fully irreducible outer automorphism with a tt map on a 6-edge graph, then the same can be said on some graph of fewer edges.
    
In light of the previous paragraph, since the stretch factor of any element of $\outt$ is the largest eigenvalue of the transition matrix of a tt representative, the algebraic degree of the stretch factor is between 1 and 5. Moreover, the stretch factor must be a Perron number by Lemma \ref{lemma:Perron}. The smallest few Perron numbers of degrees 2, 3, 4, and 5 are known. Approximately, these are 1.618, 1.325, 1.221, and 1.124 respectively (\cite[Table 3]{boydsupplement}). Since $\mathfrak{g}$ has stretch factor less than 1.32 and 1.22, the smallest stretch factor must be a degree-5 Perron number. 

The two smallest degree-5 Perron numbers are approximately $\gamma \approx 1.124$, and $\lambda \approx 1.167$, the largest real root of $x^5+x^4-x^2-x-1$ and $x^5-x-1$ respectively (\cite[Table 3]{boydsupplement}). Since $\mathfrak{g}$ has stretch factor exactly $\lambda$, we need only rule out $\gamma$ as a possible stretch factor. Notice that the trace of $x^5+x^4-x^2-x-1$ is $-1$. Thus any $5$ dimensional matrix with this characteristic polynomial must have negative trace. Since transition matrices have non-negative integer entries, this is impossible. Thus no element of $\outt$ has stretch factor $\gamma$. So $\mathfrak{g}$ has the minimal stretch factor among fully irreducible elements of $\outt$, as desired. \qedhere

\end{proof}

\vskip10pt

%%%%%%%%%%%%%%%%%%%%%%%%%%%%%%%%%%%%%%%%%%%%%%%%%%%%%%%%%%%%%%%%%%%%
%%%%%%%%%%%%%%%%%%%%%%%%%%%%%%%%%%%%%%%%%%%%%%%%%%%%%%%%%%%%%%%%%%%%

%\bigskip

\section{Uniqueness of the single-fold principal train track map}
\label{s:single_fold}

%\smallskip

\parpic[r]{\includegraphics[width=2in]{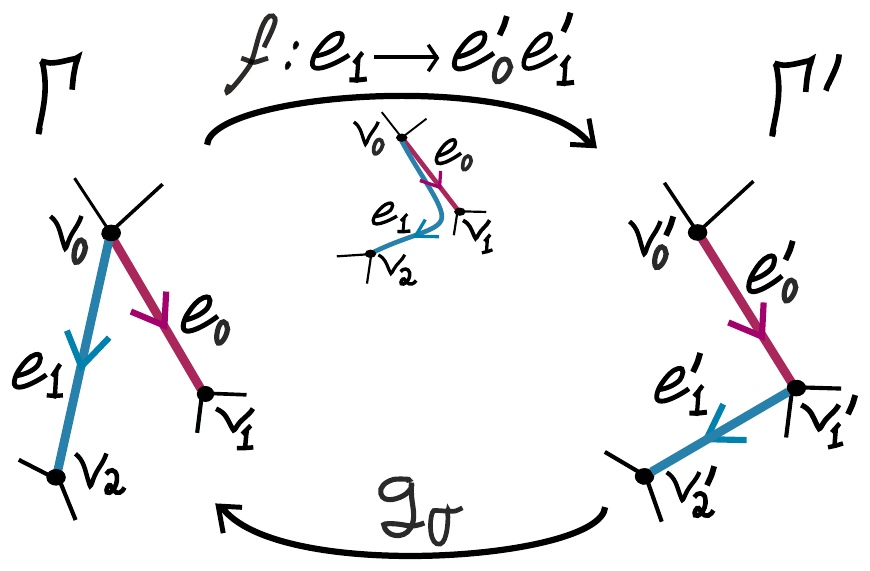}}
The goal of this section is to prove Theorem~\ref{t:unique}. 
Assume throughout that $\G$ is a graph of rank $r\geq 3$. Suppose that $h=g_{\sigma}\circ f$ is a tt representative of a principal fully irreducible $\vphi\in\out$ satisfying that $f\from\G\to\G'$ is a single proper full fold defined by $f\from e_1\mapsto e_0'e_1'$ and $g_{\sigma}\from\G'\to\G$ is a graph relabeling isomorphism. Notice that one could not have $h=g_{\sigma}\circ f$ if $f$ were instead a partial fold or a complete fold, because these types of folds change the number of vertices in a graph, and hence change the graph-isomorphism type of a graph. 

By Lemma \ref{lemma:LonelyDirection}, $\G$ will have a single valence-4 vertex and all other vertices will have valence 3. Further, the fold must occur at the unique valence-4 vertex in $\G$. Let $v_1$ denote the terminal vertex of $e_0$ and $v_2$ the terminal vertex of $e_1$, so that $e_0=[v_0,v_1]$ and $e_1 = [v_0,v_2]$. We label the vertices of $\G'$ so that $f(v) =v'$ and $e' = [v',w']$  for each $e=[v,w] \in E\G - \{e_1\}$ and $e_1' = [v_1',v_2']$.
%More concretely, denote the unique valence-4 vertex in $\G$ by $v_0$. Suppose $e_0=[v_0,v_1]$ and $e_1 = [v_0,v_2]$ are distinct directed edges in $\G$ and let $\G'$ be the directed graph defined by
%\begin{align*}
%V\G' & = \{v' \ | \ v \in V\G\} \\
%E\G' & = \{ e' = [v',w'] \ | \ e=[v,w] \in E\G - \{e_1\} \}  \cup \{e_1' = %[v_1',v_2']\}
%\end{align*}
%\noindent The graph map $f$ is defined on the vertices of $\Gamma$ by $f(v) =v'$ for each $v \in V\G$ and on the edges by
%$$f(e)= \begin{cases}e_0'e_1' & \text{if }e=e_1\\
%e' & \text{otherwise}
%\end{cases}$$
%for any $e \in E\G$.  

The following lemma will help us further set notation for the graphs $\Gamma$ and $\Gamma'$, as well as aid in the proof of Theorem~\ref{t:unique}.

\begin{lem}\label{lemma:vertices}
Suppose $\G$ is a graph with rank $r$ and $h=g_{\sigma}\circ f$ is an irreducible tt representative of a principal fully irreducible $\vphi\in\out$ such that $f\from\G\to\G'$ is a single proper full fold defined by $ f \from e_1 \mapsto e_0'e_1'$  and $g_{\sigma}\from\G'\to\G$ is a graph relabeling isomorphism. Label the vertices incident to the edges $e_0, e_1, e_1'$, and $e_0'$ as described above. Then
 \begin{enumerate}[(a)]
\item $\,|\,V\G\,|\,=2r-3$ and $\,|\,E\G\,|\,=3r-4$, and
\item the vertices $v_0, v_1,$ and $v_2$ are distinct, and
\item $g_{\sigma}(v_1')=v_0$,  $g_\sigma(e_1') = e_0$, and $g_{\sigma}(v_2') = v_1$, and
\item $h$ is transitive on the vertex set of $\G$.
\end{enumerate}
\end{lem}

\begin{proof}
We prove (a) - (d) in order, one at a time.

\smallskip
\begin{enumerate}[(a)]
\item Since the rank of $\Gamma$ is $r$, the Euler characteristic satisfies $\chi(\Gamma)=1-r$. So 
$$r-1=\,|\, E\G \,|\, - \,|\, V\G \,|\,.$$ 
Since $\G$ has a single valence-4 vertex and all other vertices are of valence 3, we have:
$$  V\G=2r-3 \quad \quad \quad \text{and} \quad \quad \quad E\G=3r-4.$$

\item If $v_0=v_1=v_2$, both $e_0$ and $e_1$ would be single-edge loops. Since valence$(v_0)=4$ and part (a) implies $\G$ must have more than one vertex,  $\G$ would be disconnected. Thus the three vertices are not all equal.  
    
\parpic[r]{\includegraphics[width=2.5in]{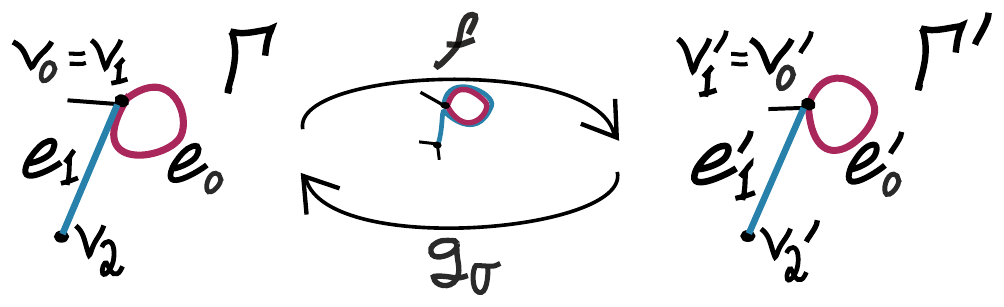}} 
If $v_0=v_1$, then $e_0$ is a single-edge loop, as is its image in $\G'$.  Since $f(v_0) = v_0'$ is the unique valence-4 vertex in $\Gamma'$, we must have $h(v_0)=g_{\sigma}(v_0')=v_0$. However, if $v_0$ is fixed by $h$, the set of edges incident to $v_0$ is $h$-invariant. Since by part (a) there are necessarily more edges in $\Gamma$ not incident to $v_0$, this contradicts the irreducibility of $h$.\\

\vspace*{-1mm}

\parpic[l]{\includegraphics[width=2.5in]{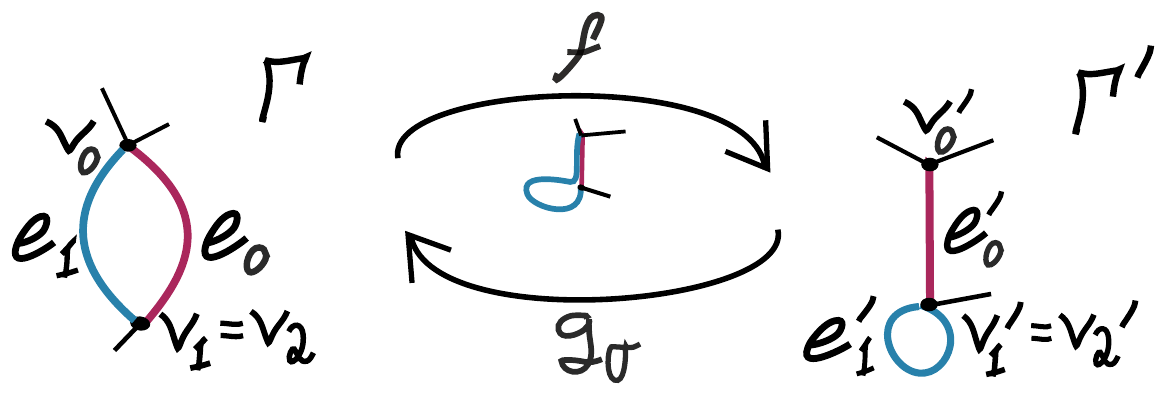}} 
 If $v_1=v_2$, then $e_1'$ is a single-edge loop in $\G'$. Now $\G'$ has one more edge which is a single-edge loop than $\G$ does, contradicting that $g_{\sigma}$ is a graph isomorphism. \\
 
 \smallskip

\parpic[r]{\includegraphics[width=2.5in]{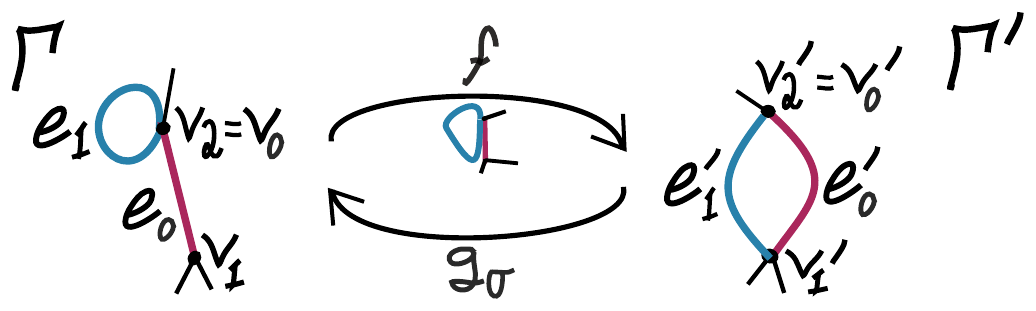}}
If $v_0=v_2$, then $e_1$ is a single-edge loop. However $e_1'$ is not a single-edge loop in $\G'$, so $\G$ has one more edge which is a single-edge loop than $\G'$, again contradicting that $g_{\sigma}$ is a graph isomorphism. \\

\vspace*{-3mm}

\item In light of (b) we have the following picture (and will show the vertex colors reflect how they are mapped):

\vspace*{-3mm}

\begin{center}
\includegraphics[width=3.75in]{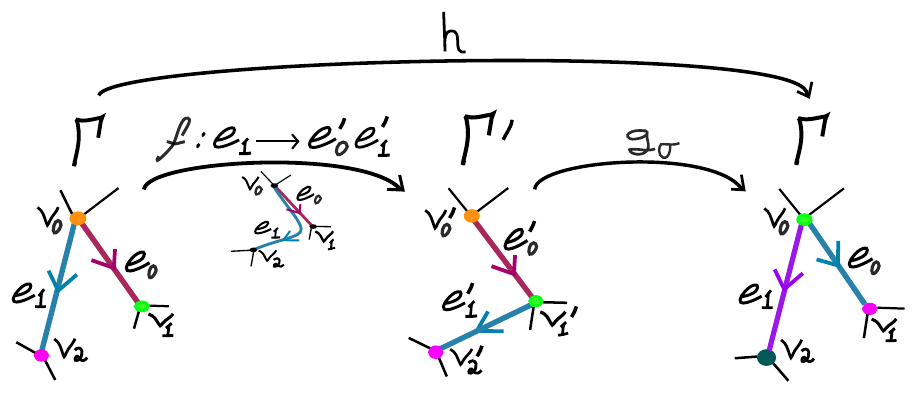}
\end{center}

Since $v_0$ is the unique valence-4 vertex of $\G$ and $v_1'$ is the unique valence-4 vertex of $\G'$ and $g_{\sigma}$ is a graph isomorphism, we must have $g_{\sigma}(v_1')=v_0$. 

By Lemma \ref{lemma:LonelyDirection}, all but one direction of $\G$ must be periodic under $h$. Hence 
\begin{equation}\label{directions}
|Dh^{k}(\mD \G)| = |\mD \G|-1
\end{equation}
for all powers $k \geq 1$. Observe that $e_1'$ is the single direction in $\G'$ that is not in the image of $Df$. 

We claim $Dg_{\sigma}(e_1') \in \{e_0,e_1\}$. Suppose not. Since $g_\sigma$ is a graph isomorphism, $Dg_{\sigma}$ is bijective on the set of directions of $\G'$. Hence there exists a pair of distinct directions $a',b' \in \mD \G'$ such that $Dg_{\sigma}(a')=e_1$ and $Dg_{\sigma}(b')=e_0$. Since $g_{\sigma}(v_1')=v_0$ and $g_{\sigma}$ is a graph-isomorphism, $a'$ and $b'$ emanate from $v_1'$. Since we assumed neither $a'$ nor $b'$ are equal to the $Df$-unachieved direction $e_1'$, the definition of $\G'$ ensures there exists a pair of directions $a,b \in \mD \G$ so that $Df(a)=a'$ and $Df(b)=b'$. Further, since $a'$ and $b'$ are incident to $v_1'$ and $f(v_1)=v_1'$, we have that $a$ and $b$ are incident to $v_1$. Thus,

 $$Dh^2(a) = Dg_{\sigma}(Df(Dg_{\sigma}(Df(a)))) = Dg_{\sigma}(Df(Dg_{\sigma}(a'))) = Dg_{\sigma}(Df(e_1)) = Dg_{\sigma}(e_0') $$
 and
  $$Dh^2(b) = Dg_{\sigma}(Df(Dg_{\sigma}(Df(b)))) = Dg_{\sigma}(Df(Dg_{\sigma}(b')))  = Dg_{\sigma}(Df(e_0)) = Dg_{\sigma}(e_0').$$
  
 \bigskip
 
\noindent Hence  $Dh^{2}(a) = Dh^{2}(b)$, which implies $|Dh^{2}(\mD \G)| \leq |\mD \G|-2$, contradicting  \ref{directions}. Thus, $Dg_{\sigma}(e_1') \in \{e_0,e_1\}$, as desired.

Since $g_{\sigma}$ is a graph isomorphism, the previous paragraph implies that $g_{\sigma}(e_1') \in \{e_1,e_0\}$. Suppose for the sake of contradiction that $g_{\sigma}(e_1')=e_1$. Then $g_{\sigma}(v_2')=v_2$, and so $h$ fixes the vertex $v_2$. Since $Dh(\ol{e_1}) =  Dg_{\sigma}(Df(\ol{e_1}))=Dg_{\sigma}(\ol{e_1}')=\ol{e_1}$, the set consisting of the remaining two edges incident to $v_2$ is invariant under $h$. This contradicts the irreducibility of $h$. Therefore $g_{\sigma}(e_1') = e_0$. 
 
Since $e_0=[v_0,v_1]$ and  $e_1'=[v_1',v_2']$ and $g_{\sigma}$ is a graph isomorphism, and we already have $g_{\sigma}(v_1')=v_0$, we know that $g_{\sigma}(v_2')$ must be the vertex of $e_0$ other than $v_0$, which is exactly $v_1$. This completes the proof of part (c). 

\smallskip

\item Suppose $h$ is not transitive on the vertex set of $\G$. Then $V\G = X \sqcup Y$ for nonempty proper subsets $X$ and $Y$, both $h$-invariant. Without loss of generality, suppose $v_2 \in X$. Let $E_Y$ be the edges in $\G$ with at least one incident vertex in $Y$. We aim to show $E_Y$ is a nonempty proper subset of $E\G$ which is invariant under $h = g_{\sigma} \circ f$, contradicting that $h$ must be an \textit{irreducible} tt map.

By part (c), $h(v_2)=g_{\sigma}(v_2')=v_1$ and  $h(v_1) = g_{\sigma}(v_1')=v_0$ . Hence both  $v_1,v_0 \in X$ by invariance of $X$.  As $Y$ is disjoint from $X$, all three vertices $v_0,v_1,v_2 \notin Y$, implying $e_0, e_1 \notin E_Y$. This guarantees that $E_Y$ is indeed a proper subset of the edge set. Clearly $E_Y$ is nonempty, since $Y$ is nonempty and every vertex in $\G$ has nonzero valence. All that remains to show is that $E_Y$ is $h$-invariant. 

Let $e \in E_Y$. By the definition of $E_Y$, there is a vertex $v \in Y$ incident to $e$. Since $e_1 \notin E_Y$, we know $e \neq e_1$, so by the definition of $\G'$, we have $e'$ is incident to $v'$ in $\Gamma'$. By the definition of $f$, 
$$h(e) = g_{\sigma}(f(e)) = g_{\sigma}(e')$$
and
$$h(v) = g_{\sigma}(f(v))= g_{\sigma}(v').$$
Since $g_{\sigma}$ is a graph isomorphism, $g_{\sigma}(e')$ remains incident to $g_{\sigma}(v')$. Moreover, $g_{\sigma}(v')=h(v)  \in  Y$ by invariance of $Y$. Hence $h(e) = g_{\sigma}(e') \in E_Y$, so $E_Y$ is $h$-invariant, completing the proof of (d). \qedhere
\end{enumerate} 
\end{proof}

 \begin{thmx}\label{t:unique}
Up to edge relabeling, the map $\mathfrak{g}$ of Figure \ref{f:map3} is the only train track map representing a principal fully irreducible outer automorphism in any rank whose Stallings fold decomposition consists of only a single fold composed with a graph-relabeling isomorphism.
\end{thmx}

\begin{proof}
Suppose $h$ is a tt representative of a principal fully irreducible $\vphi \in \outt$. Then, by Proposition \ref{lemma:map3}, $h$ is represented by a directed loop in $\widehat{\mA_3}$ including the ltt structure (call it $G$) at Node I. If $h$ contains only a single fold, then this fold must be one of the folds taking $G$ to itself. Up to permutation-relabeling of ltt structures, there are only 2 such folds (the left-most 2 folds in the Node I box in Figure \ref{f:mA3}). $\mathfrak{g}$ is indicated using gold (dashed) arrows in Figure \ref{f:mA3}. In the graph-labeling class of $\widehat{\mA_3}$ included in Figure \ref{f:mA3}, the other fold within the Node I box would lead to a fold of $d$ over $\bar{c}$ and then the graph isomorphism $a\mapsto\bar{e}$, $b\mapsto c$, $c\mapsto b$, $d\mapsto d$, $e\mapsto\bar{a}$ (a combination that leaves invariant $\{b\} \cup \{c\}$).  
%and the other combination of a fold then graph-relabeling isomorphism leaves invariant the subgraph labeled with $b$ and $c$ in the graph-labeling class of $\widehat{\mA_3}$ included in Figure \ref{f:mA3}. 
Thus, up to edge relabeling, $h$ is unique.

We now show that for each $r\geq 4$ no such map can exist. Let $n= 2r-3$ be the number of vertices in $\G$. By Lemma ~\ref{lemma:vertices}(d), the vertices of $\Gamma$ are all in the same orbit under $h$. Hence we can recursively label the vertices of $\Gamma$ as follows, taking subscripts modulo $n$:
 \begin{enumerate}
\item Let $v_0$ remain as is and
\item let $v_{k-1}:=h(v_k)$ for each $1 \leq k \leq n$.
\end{enumerate} 
On the vertex set, $h$ is now given by: 
$$v_0 \mapsto v_{n-1} \mapsto v_{n-2} \mapsto \dots \mapsto v_2 \mapsto v_1 \mapsto v_0.$$

We maintain the convention that a vertex $v$ in $\G$ is given the label $v'$ in $\G'$, so we have a labeling of vertices in $\G'$ as well. Since $f(v)=v'$ for each $v \in V\G$, the above labeling implies $g_{\sigma}(v_k')=v_{k-1}$ and $g_{\sigma}^{-1}(v_k)=v'_{k+1}$. Observe that by Lemma ~\ref{lemma:vertices}(c) our labeling is consistent with our original labeling of $v_1$ and $v_2$. 

We will reach a contradiction to the existence of $h$ by showing $v_2$ would have the wrong valence. In what follows, we rely on the fact that the edge-labeling induced by $f$ (in the sense of \S \ref{ss:StallingsFoldDecompositions}) provides a bijection between $E\G\backslash\{e_1\}$ and $E\G'\backslash\{[v_1',v_2']\}$ given by mapping $e=[v,w]$ to  $e'=[v',w']$. Thus, if $e_j'$ and $e_k'$ have distinct vertices and are not $e_1'=[v_1',v_2']$ or $\ol{e_1'}=[v_2',v_1']$, then removing primes on vertices gives distinct edges of $\G$.

Let $e_2':=g_{\sigma}^{-1}(e_1)$. Since $e_1=[v_0,v_2]$ and $g_{\sigma}^{-1}$ is a graph isomorphism, 
$$e_2'=g_{\sigma}^{-1}(e_1)=[g_{\sigma}^{-1}(v_0),g_{\sigma}^{-1}(v_2)]=
[v_1',v_3'].$$
Since $e_1'=[v_1',v_2']$ and $v_2' \neq v_3'$, this tells us $e_2' \notin \{e_1',\overline{e_1'}\}$. Therefore, by our convention for labeling the vertices of $\G'$, there is an edge in $E \G$ 
%distinct from $\{e_0,e_1\}^{\pm 1}$ and 
joining $v_1$ and $v_3$. We call this edge $e_2$. Similarly, let  $e_3':=g_{\sigma}^{-1}(e_2)$. Since $e_2=[v_1,v_3]$ and $g_{\sigma}^{-1}$ is a graph isomorphism,
$$e_3' = g_{\sigma}^{-1}(e_2) = [g_{\sigma}^{-1}(v_1),g_{\sigma}^{-1}(v_3)]=[v_2',v_4'].$$
By Lemma ~\ref{lemma:vertices}(a) and the assumption that $r \geq 4$, we have $|V\G'|=|V\G|\geq 5$. Hence $v_1' \neq v_4'$, so $e_3' \notin \{e_1',\overline{e_1'}\}$. So there is an edge in $E\G$ joining $v_2$ and $v_4$. We call this edge $e_3$.

Let $\alpha_1:=g_{\sigma}(e_0')$. Since $e_0'=[v_0',v_1']$ and $g_{\sigma}$ is a graph isomorphism,
$$\alpha_1 = g_{\sigma}(e_0') = [g_\sigma(v_0'), g_\sigma(v_1')] = [v_{n-1},v_{0}].$$
Since $n \geq 5$, we know $v_{n-1} \neq v_2$, and hence also $\alpha_1 \notin \{e_1,\overline{e_1}\}$. By the labeling of $\G'$, there is an edge $\alpha_1' = [v_{n-1}',v_{0}'] \in E\G'$.  Similarly, let 
$$\alpha_2:=g_{\sigma}(\alpha_1') = [g_\sigma(v_{n-1}'), g_\sigma(v_{0}')] =[v_{n-2},v_{n-1}].$$ 
Again, by $\G'$ labeling conventions, there is an edge $\alpha_2' = [v_{n-2}',v_{n-1}'] \in E\G'$. Recursively define 
$$\alpha_k := g_{\sigma}(\alpha_{k-1}') = [g_{\sigma}(v_{n-k+1}'),g_{\sigma}(v_{n-k+2}')]=[v_{n-k}',v_{n-k+1}'].$$
for each $k \in \{2, \dots n-1\}$.

The final two recursively defined edges $\alpha_{n-2} = [v_2,v_3]$ and $\alpha_{n-1}=[v_1,v_2]$ both contain $v_2$. Moreover, $e_1 = [v_0,v_2]$ and $e_3=[v_2,v_4]$ contain $v_2$. Since $v_0, v_1,v_3,$ and $v_4$ are distinct vertices, all four of these edges are distinct. Therefore $v_2$ is contained in at least 4 distinct edges in $\G$. This contradicts that $v_2$ should have valence 3. Thus for each $r \geq 4$ there cannot exist a single-fold irreducible tt map $h$ representing a principal fully irreducible element of $\out$. \qedhere

%\vspace*{-5mm}

\begin{center}
\begin{figure}[t]
\includegraphics[width=5in]{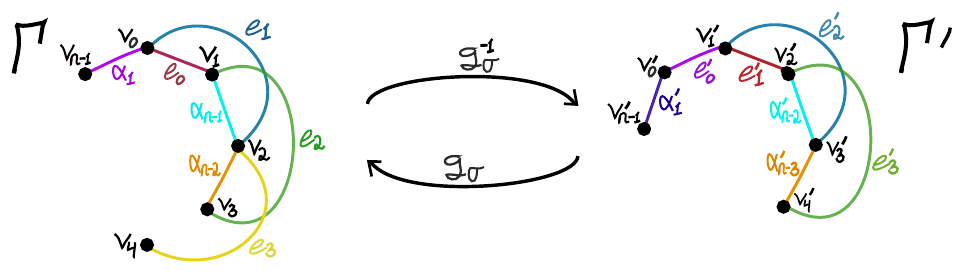}
\end{figure}
\end{center}

\end{proof}
\vfill

 \begin{center}
\begin{figure}[H]
\includegraphics[height=8.5in]{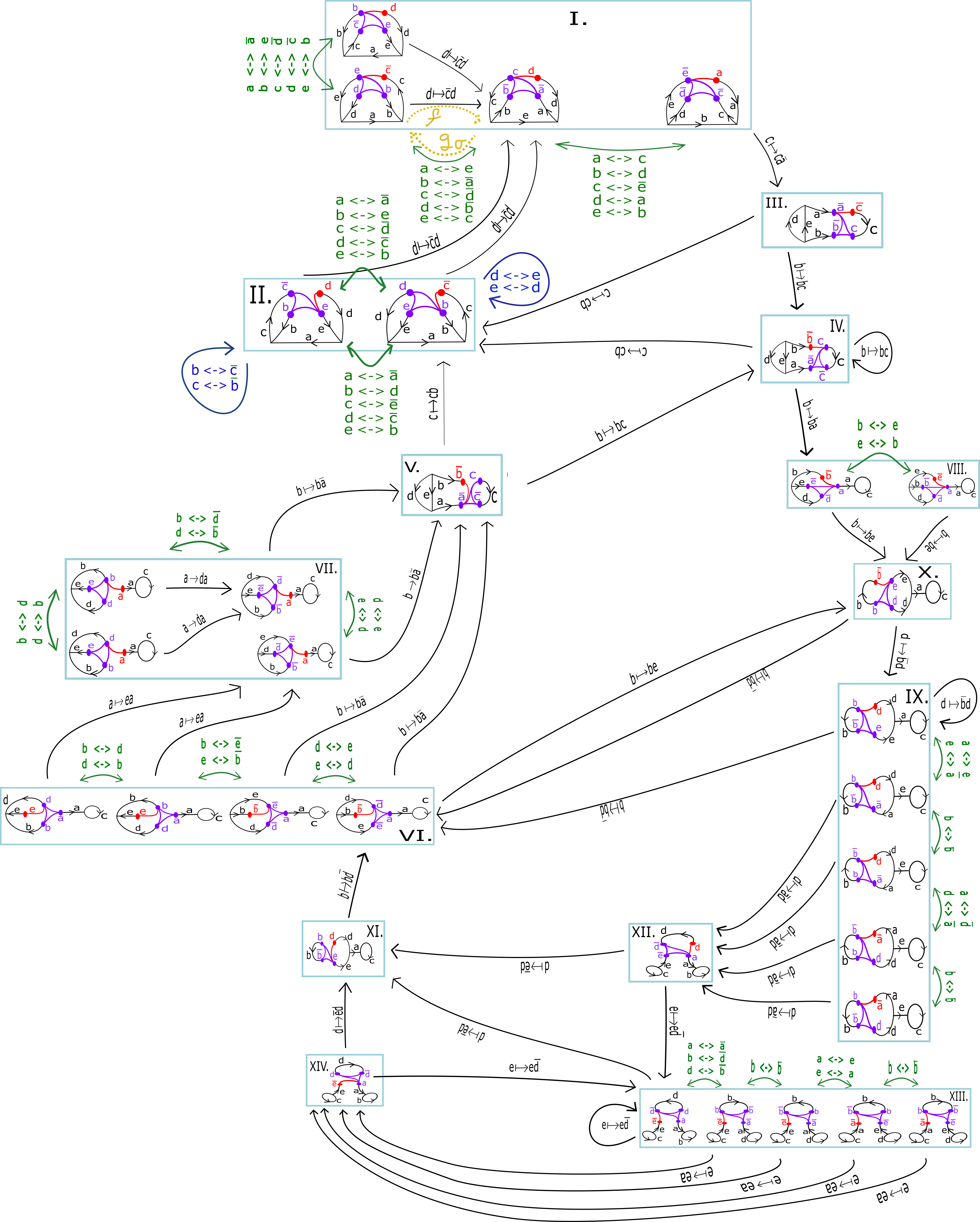}
\vspace*{5mm}
 \caption{This figure depicts the Rank-3 Principal Stratum Automaton. Permutations of the automata are in green; note that if moving right to left across the page one must read the permuatation right to left as well. Compositions of included permutations, as well as color-preserving graph symmetries, are implicitly included. (Symmetries are included in blue at II as an example.) The map of Figure \ref{f:map3} is in dotted gold.} \label{f:mA3}
\end{figure}
\end{center}

\newcommand{\etalchar}[1]{$^{#1}$}

\end{document}